\documentclass[11pt]{article}

\usepackage[utf8]{inputenc}
\usepackage[T1]{fontenc}
\usepackage[a4paper,margin=3cm]{geometry}
\usepackage{amsmath,amssymb,amsfonts,amsthm}
\usepackage{lmodern}
\usepackage{microtype}
\usepackage{booktabs}
\usepackage{array}
\usepackage{placeins}
\usepackage{url}
\usepackage[colorlinks=true,linkcolor=blue,citecolor=blue,urlcolor=blue]{hyperref}
\hypersetup{pdftitle={The Point-Cross Dimension},pdfauthor={Nadir Maaroufi}}

\numberwithin{equation}{section}

\theoremstyle{plain}
\newtheorem{theorem}{Theorem}[section]
\newtheorem{proposition}[theorem]{Proposition}
\newtheorem{lemma}[theorem]{Lemma}
\newtheorem{corollary}[theorem]{Corollary}

\theoremstyle{definition}
\newtheorem{definition}[theorem]{Definition}
\newtheorem{example}[theorem]{Example}
\newtheorem{remark}[theorem]{Remark}

\DeclareMathOperator{\Tan}{Tan}
\DeclareMathOperator{\Span}{span}

\newcommand{\R}{\mathbb{R}}
\newcommand{\Pbox}{\dim_{\mathrm{Pbox}}}
\newcommand{\dimcross}{\dim_{\times}}
\newcommand{\Pvec}{\dim_{\mathrm{Pvec}}}
\newcommand{\Ptan}{\dim_{\mathrm{Ptan}}}
\newcommand{\dimBox}{\overline{\dim}_{B}}

\providecommand{\subjclass}[1]{%
  \par\medskip\noindent\textbf{2020 Mathematics Subject Classification.} #1}
\providecommand{\keywords}[1]{%
  \par\noindent\textbf{Keywords.} #1}

\title{The Point-Cross Dimension}

\author{Nadir Maaroufi\\
International University of Rabat, TICLab, Sala Al Jadida, 11100, Morocco\\
\texttt{nadir.maaroufi@uir.ac.ma}}

\date{}

\begin{document}

\maketitle

\begin{abstract}
Two transverse segments have local Hausdorff and upper box dimension one at their
crossing, although they carry two independent tangent directions. We define
the Point-Cross Dimension by assigning each tangent line the largest local
upper box dimension captured by intersection with a regular \(C^1\) curve
in that direction, then taking the supremum of their sums over finite
independent families of tangent lines. Every Bouligand direction admits such
a recurrent curve. The resulting profile is
invariant under closure and covariant under local \(C^1\)-diffeomorphisms.
Each directional weight is attained, and the profile is upper semicontinuous.
An integral representation by ranks of weight superlevel sets gives
submodularity and, together with projection bounds, exact Cartesian-product
additivity. Germs with the same projective tangent support, local upper box
dimension, and rank of segment directions can have different Point-Cross
Dimensions. Further computations give exact directional profiles for chirps
and the Point-Cross Dimension of a self-affine carpet. Finally, approximation
of a ray with error \(O(t^d)\), for \(d>1\), forces directional weight at least
\(1-1/d\), and this bound is sharp.

\end{abstract}


\setcounter{tocdepth}{1}
\tableofcontents

\section{Introduction}
Two transverse line segments have local Hausdorff and upper box dimension
one at their intersection. The union of the coordinate axes in \(\R^n\)
also has dimension one at the origin, by the finite-union maximum rule.
Their tangent directions, however, span spaces of dimensions two and \(n\),
respectively. We seek a quantity that records these independent directions
and allows a fractional contribution when a segment is replaced by a fractal
set.

The Bouligand tangent support records which limiting secant directions occur.
Its rank counts independent directions, but does not measure the complexity
carried by each one. Local upper box dimension measures covering complexity
near the point without assigning it to directions. The question here is how
much of that complexity a regular curve can capture when its tangent
direction is prescribed.

Let \(A\subset\R^n\) and \(x\in\overline A\). Write
\(\mathbb P\Sigma_x(A)\) for the lines spanned by Bouligand directions at
\(x\), and call these lines \emph{channels}. A probe in a channel \(\xi\)
is a regular one-sided \(C^1\) curve germ through \(x\), tangent to
\(\xi\), which meets \(A\setminus\{x\}\) arbitrarily close to \(x\).
Its trace need not lie in \(A\). With \(\Pbox\) denoting local upper box
dimension, define
\[
 \theta_x^A(\xi)=\sup_{\gamma\in\mathcal G_x^A(\xi)}
       \Pbox\{x\}_{A\cap\Gamma_\gamma},
 \qquad
 \dimcross\{x\}_A=\sup_{\Xi\in\mathcal I_x(A)}
       \sum_{\xi\in\Xi}\theta_x^A(\xi),
\]
where \(\mathcal I_x(A)\) consists of finite families of linearly
independent channels. Opposite orientations belong to the same channel.
The weights lie in \([0,1]\). Thus the scalar quantity lies between the
rank \(\Pvec\) of segment directions and the rank \(\Ptan\) of the
Bouligand tangent cone.

The choice of probes is tied to the category of invariance. Regular \(C^1\)
curve germs are exactly local \(C^1\)-diffeomorphic images of one-sided line
germs. They realize every Bouligand orientation, whereas no fixed
\(C^{1,\alpha}\) class has this property for arbitrary sets. Their local
chords become aligned with the prescribed direction. Controlling only
secants from the base point would permit a Lipschitz probe to collect
transverse complexity, as the example in
Subsection~\ref{subsec:lipschitz-secant-defect} shows.

\paragraph{Main results.}
Sequential reconstruction gives an optimizing probe for every directional
weight and proves upper semicontinuity of the zero-extended profile.
Perturbations of sequential witnesses yield closure invariance of the full
weighted support. Under a local \(C^1\)-diffeomorphism \(\Phi\), the
derivative transports each channel and preserves its weight. Locally
subanalytic germs have weight one on every tangent channel.

The aggregation has a second description. Let \(r_x^A(t)\) be the dimension
of the span of channels whose weights are at least \(t\). Then
\[
 \dimcross\{x\}_A=\int_0^1r_x^A(t)\,dt.
\]
This formula counts independent directions at each complexity threshold.
It does not require a maximizing independent family, which need not exist.
Writing \(D_x(E)=\dimcross\{x\}_E\), we prove
\[
 D_x(A\cup B)+D_x(A\cap B)\le D_x(A)+D_x(B),
\]
and, for \(x\in\overline A\) and \(y\in\overline B\),
\[
 \dimcross\{(x,y)\}_{A\times B}
 =\dimcross\{x\}_A+\dimcross\{y\}_B.
\]
The product identity follows from an exact direct-sum formula for the
subspaces generated at each threshold. A further geometric estimate gives
\(\theta_x^A([v])\ge1-1/d\) whenever a unit ray satisfies
\(\operatorname{dist}(x+tv,A)=O(t^d)\), with \(d>1\).
Power-law sequences attain this bound for every \(d>1\).

The additional information is visible in a simple family. For each
\(0<d<1\), let \(K_d\) be the two-map Cantor set of dimension \(d\)
used in Example~\ref{ex:point-cross-non-determination}, and set
\(A_d=[0,1]e_1\cup K_d e_2\). Its projective tangent support is always
\(\{[e_1],[e_2]\}\), while
\[
 \Pvec\{0\}_{A_d}=\Pbox\{0\}_{A_d}=1,
 \qquad \dimcross\{0\}_{A_d}=1+d.
\]
Thus even the full projective tangent support together with these two scalar
values does not determine Point-Cross. Chirps yield exact profiles through
intersection counts with regular curves. A Bedford\textendash{}McMullen carpet yields
\(\Pvec<\Pbox<\dimcross<\Ptan\) through anisotropic scaling and projection
bounds.

\paragraph{Related constructions.}
Koike and Paunescu study the dimension of common direction sets and its
behavior under bi-Lipschitz equivalence in the subanalytic
category~\cite{KoikePaunescu2009}. Direction sets, sequence selection, tangent
cones, and directional bundles also occur in
\cite{KoikeLoiPaunescuShiota2013,KoikePaunescu2015,Sampaio2016,KoikePaunescu2024}.
These works concern the occurrence and geometry of directions. Here a local
covering weight is attached to each individual channel. Our quantitative
accessibility hypothesis is expressible using the sea-tangle neighborhoods
of~\cite{KoikePaunescu2009}, with a conclusion about this weight.

The tangential dimensions of Guido and Isola concern two-scale local metric
behavior~\cite{GuidoIsolaI,GuidoIsolaII}. Tangents of invariant sets can
encode extremal scaling data~\cite{KaenmakiRutar2023}, and pointwise tangent
fields constrain the tangents of rectifiable curves almost everywhere along
their intersections~\cite{DavidErikssonBiqueSchul2026}. Alberti
representations decompose a measure along directional curve
families~\cite{Bate2015}. Our construction instead assigns weights at a
prescribed point without a reference measure. Countable trace intersections
may have positive weight.

The local upper box quantity comes from the Point-Extended Box
framework~\cite{pointbox,pointcsf}. Its sequential realization uses finite
separated blocks, as in the upper box sequence-property arguments
of~\cite{Ivanov2023,FalconerLiu2026}. We give the prescribed-point version
needed here. The geometric work is to preserve those blocks under changes
of direction and perturbations while reconstructing a regular curve.
The threshold integral is the finite-rank weighted-independence principle
underlying matroid optimization~\cite{Edmonds1971}, applied to the geometric
profile.

Sections~\ref{sec:exact-tangential-ranks}\textendash{}\ref{sec:point-cross-dimension}
introduce tangent support, probes, and the weighted invariant.
Sections~\ref{sec:closure-shadowing} and~\ref{sec:structural-properties}
establish reconstruction, stability, and saturation.
Sections~\ref{sec:directional-calculus} and~\ref{sec:quantitative-accessibility}
give the set calculus and the sharp accessibility bound.
Section~\ref{sec:discriminating-examples} computes the examples, followed by
the questions in Section~\ref{sec:outlook-open-problems}.

\section{Directional ranks and tangent support}
\label{sec:exact-tangential-ranks}

Two unweighted directional ranks frame the weighted construction. The
Point-Vector Dimension uses straight segment germs and is neither invariant
under closure nor under \(C^1\)-changes of coordinates. The Point-Tangential
Dimension uses the Bouligand tangent cone, is stable under both operations,
and supports the weighted theory.

Throughout, we use the convention \(\Span(\varnothing)=\{0\}\).

\subsection{The point-vector dimension}
\label{subsec:exact-directional-rank}

For \(A\subset\R^n\) and \(x\in\R^n\), set
\[
    \operatorname{Dir}_A(x)
    :=
    \bigl\{v\in\R^n\setminus\{0\}:
    \exists\delta>0\text{ such that }
    \{x+tv:0<t<\delta\}\subset A\bigr\}.
\]
The point-vector dimension is
\[
    \Pvec\{x\}_A
    :=
    \dim\Span\bigl(\operatorname{Dir}_A(x)\bigr).
\]
Thus \(\Pvec\) is the maximal number of independent directions realized by
one-sided segment germs. The base point need not belong to \(A\), but the
existence of an exact direction forces \(x\in\overline A\).

\subsection{The point-tangential dimension}
\label{subsec:tangential-rank}

Following Bouligand~\cite{bouligand}, for \(A\subset\R^n\) and
\(x\in\R^n\), define the tangent cone by
\[
    \Tan_x(A)
    :=
    \left\{v\in\R^n:
    \exists a_j\in A,\ a_j\to x,\ \exists\lambda_j\downarrow0,
    \frac{a_j-x}{\lambda_j}\to v
    \right\}.
\]
Set
\[
    \Sigma_x(A):=\Tan_x(A)\cap S^{n-1},
    \qquad
    \Ptan\{x\}_A:=\dim\Span\bigl(\Tan_x(A)\bigr).
\]
The cone is empty when \(x\notin\overline A\). Otherwise it contains \(0\).
It is positively homogeneous, so
\(\Span(\Tan_x(A))=\Span(\Sigma_x(A))\) and consequently
\(\Ptan\{x\}_A=\dim\Span(\Sigma_x(A))\).

For \(v\in S^{n-1}\), the definition has the following equivalent
normalized-secant form:
\[
    v\in\Sigma_x(A)
    \quad\Longleftrightarrow\quad
    \frac{y_j-x}{\|y_j-x\|}\longrightarrow v
    \quad\text{for some }y_j\in A\setminus\{x\}\text{ with }y_j\to x.
\]
Indeed, the forward implication gives
\(\|y_j-x\|/\lambda_j\to1\), while the converse follows by taking
\(\lambda_j=\|y_j-x\|\) along a decreasing subsequence.

\begin{proposition}[Basic properties of the tangent cone]
\label{prop:exact-tangential-ranks}
Let \(A,B\subset\R^n\) and \(x\in\R^n\). Then:
\begin{enumerate}
    \item
    If \(A\subset B\), then \(\Tan_x(A)\subset\Tan_x(B)\).
    \item
    \(\Tan_x(A\cap U)=\Tan_x(A)\) for every neighborhood \(U\) of \(x\).
    \item
    \(\Tan_x(A)=\Tan_x(\overline A)\).
    \item
    \(\operatorname{Dir}_A(x)\subset\Tan_x(A)\setminus\{0\}\), hence
    \(\Pvec\{x\}_A\le\Ptan\{x\}_A\).
\end{enumerate}
Consequently, \(\Ptan\) is monotone, local, and invariant under closure.
\end{proposition}

\begin{proof}
Monotonicity and locality follow directly from the sequential definition. For
closure invariance, let \(y_j\in\overline A\) realize a tangent vector and choose
\(a_j\in A\) with \(\|a_j-y_j\|\le\lambda_j^2\). Since
\(\|a_j-y_j\|/\lambda_j\le\lambda_j\to0\), the normalized quotients have the
same limit. This proves \(\Tan_x(\overline A)\subset\Tan_x(A)\), while the
reverse inclusion follows from monotonicity. Finally, an exact direction is
realized by taking \(a_j=x+\lambda_jv\).
\end{proof}

\begin{proposition}[Tangent cones under \(C^1\) changes of coordinates]
\label{prop:ptan-c1-covariance}
Let \(U,V\subset\R^n\) be open, let \(x\in U\), let \(A\subset U\), and let
\(\Phi:U\to V\) be a \(C^1\)-diffeomorphism. Then
\[
    \Tan_{\Phi(x)}(\Phi(A))
    =
    D\Phi(x)\Tan_x(A).
\]
Consequently,
\(\Ptan\{\Phi(x)\}_{\Phi(A)}=\Ptan\{x\}_A\).
\end{proposition}

\begin{proof}
Let \(a_j\) and \(\lambda_j\) realize \(v\in\Tan_x(A)\). Write
\(\Phi(a_j)-\Phi(x)=D\Phi(x)(a_j-x)+r_j\), where
\(\|r_j\|=o(\|a_j-x\|)\).
Since \((a_j-x)/\lambda_j\to v\), the quotient
\(\|a_j-x\|/\lambda_j\) remains bounded, and hence
\(r_j/\lambda_j\to0\). Therefore
\[
    \frac{\Phi(a_j)-\Phi(x)}{\lambda_j}
    =
    D\Phi(x)\frac{a_j-x}{\lambda_j}+\frac{r_j}{\lambda_j}
    \longrightarrow D\Phi(x)v.
\]
Applying the same argument to \(\Phi^{-1}\) proves equality of the tangent
cones. Since \(D\Phi(x)\) is invertible, it preserves the dimension of their
linear spans.
\end{proof}

\begin{remark}[Instability of the point-vector dimension]
The point-vector dimension is neither closure invariant nor \(C^1\)-invariant.
For example, \((\mathbb Q\cap(0,1))e_1\) has point-vector dimension \(0\) at
the origin, while its closure has point-vector dimension \(1\). Likewise,
\(\Phi(u,v)=(u,v+u^2)\) sends the germ of the \(u\)-axis, whose point-vector
dimension at the origin is \(1\), to a parabolic germ whose point-vector
dimension there is \(0\). Thus \(\Pvec\) records exact straight structure,
whereas the Bouligand support retains the limiting directions preserved by
closure and transported under \(C^1\) changes of coordinates.
\end{remark}

\subsection{Projective formulation of the point-tangential dimension}
\label{subsec:projective-tangent-support}

Let \(\pi:S^{n-1}\to\mathbb{RP}^{n-1}\), \(\pi(v)=[v]\), be the canonical
quotient. The projective tangent support of \(A\) at \(x\) is
\[
    \mathbb P\Sigma_x(A)
    :=
    \pi\bigl(\Sigma_x(A)\bigr)
    =
    \bigl\{[v]:v\in\Tan_x(A)\setminus\{0\}\bigr\}.
\]
Its elements are called the projective tangent channels of \(A\) at \(x\).
Projectivization identifies opposite representatives as one linear channel.
The orientations of a channel \(\xi\) actually realized by \(A\) are the
vectors \(v\in\Sigma_x(A)\) satisfying \([v]=\xi\). When regular \(C^1\)
probes are introduced in the next section, their initial derivatives will
select one of these orientations.

A finite family \(\Xi\subset\mathbb P\Sigma_x(A)\) is
\emph{projectively independent} if some, and hence every, choice of nonzero
representatives \(v_\xi\in\R^n\) satisfying \([v_\xi]=\xi\), \(\xi\in\Xi\),
is linearly independent. Denote the collection of these families by
\[
    \mathcal I_x(A)
    :=
    \bigl\{\Xi\subset\mathbb P\Sigma_x(A):
    \Xi\text{ is finite and projectively independent}\bigr\}.
\]
The empty family is included in \(\mathcal I_x(A)\).
The largest cardinality of a projectively independent family is
\(\dim\Span(\Tan_x(A))\). Hence
\[
    \Ptan\{x\}_A
    =
    \max_{\Xi\in\mathcal I_x(A)}|\Xi|.
\]

These tangent-cone properties pass immediately to the projective support: it
is monotone, local, and invariant under closure. Under
the hypotheses of Proposition~\ref{prop:ptan-c1-covariance}, the map
\([v]\mapsto[D\Phi(x)v]\) is a bijection from
\(\mathbb P\Sigma_x(A)\) onto
\(\mathbb P\Sigma_{\Phi(x)}(\Phi(A))\) and preserves projective independence.
It therefore induces a bijection
\(\mathcal I_x(A)\to\mathcal I_{\Phi(x)}(\Phi(A))\).

Thus \(\Ptan\) optimizes the unweighted sum
\(\sum_{\xi\in\Xi}1\) over \(\mathcal I_x(A)\). Point-Cross will use the same
support and independent families, replacing each unit contribution by a local
trace-complexity weight obtained from regular \(C^1\) probes.

\section{Regular \texorpdfstring{\(C^1\)}{C1} probes and local trace complexity}
\label{sec:c1-probes-trace-complexity}

We need curves that realize tangent directions while controlling the chords
between nearby contacts. This section gives their graph form and realization
theorem, then records the local covering tools used to reconstruct them.

\subsection{Regular \texorpdfstring{\(C^1\)}{C1} probes}
\label{subsec:regular-c1-probes}

A probe must have a genuine first-order direction and remain in the same class
under local \(C^1\) changes of coordinates. Regular \(C^1\) curve germs meet
both requirements. The next subsection expresses this choice geometrically
through vanishing transverse slope and shows that these germs are precisely
the local \(C^1\)-diffeomorphic images of one-sided line germs.

\begin{definition}[Regular \(C^1\) probe]
\label{def:regular-c1-probe}
Let \(A\subset\R^n\), let \(x\in\overline A\), and let \(v\in S^{n-1}\). A
\emph{regular \(C^1\) probe for \(A\) at \(x\) in the oriented direction
\(v\)} is an injective map \(\gamma:[0,\delta]\to\R^n\), for some
\(\delta>0\), such that
\[
    \begin{gathered}
        \gamma\in C^1([0,\delta],\R^n),
        \qquad
        \gamma(0)=x,
        \qquad
        \gamma'(t)\ne0 \quad (0\le t\le\delta),
        \\
        \gamma'(0)=\lambda v
        \quad\text{for some }\lambda>0,
    \end{gathered}
\]
and
\[
    A\cap\gamma((0,\eta])\ne\varnothing
    \quad\text{for every }\eta\in(0,\delta].
\]
The family of these probes is denoted by \(\mathcal G_x^A(v)\). The geometric
trace of \(\gamma\) is \(\Gamma_\gamma:=\gamma([0,\delta])\).
\end{definition}

No regularity is imposed on \(A\cap\Gamma_\gamma\), and recurrence means
\(x\in\overline{(A\cap\Gamma_\gamma)\setminus\{x\}}\). Restriction and
increasing regular \(C^1\) reparametrization preserve admissibility. Requiring
\(\gamma'\ne0\) on the full interval causes no loss after restriction. For
\(\xi\in\mathbb{RP}^{n-1}\), set
\[
    \mathcal G_x^A(\xi)
    :=
    \bigcup_{\substack{v\in S^{n-1}\\ [v]=\xi}}
    \mathcal G_x^A(v).
\]
If both orientations are admissible, both enter this unoriented family.

\subsection{Regular \texorpdfstring{\(C^1\)}{C1} curve germs and
vanishing transverse slope}
\label{subsec:vanishing-transverse-slope}

The derivative condition has a vanishing-slope graph representation that
controls all short chords of the trace, not only the secants joining its points
to \(x\).

\begin{proposition}[Vanishing-slope graph representation]
\label{prop:c1-vanishing-slope-normal-form}
Let \(\gamma:[0,\delta]\to\R^n\) be a regular \(C^1\) curve with
\(\gamma(0)=x\) and \(\gamma'(0)=\lambda v\), where \(\lambda>0\) and
\(v\in S^{n-1}\). After restriction and an increasing \(C^1\)
reparametrization, it has the form
\[
    \gamma(s)=x+sv+f(s),
    \qquad
    f(s)\in v^\perp,
\]
where \(f\in C^1([0,\delta],v^\perp)\), \(f(0)=f'(0)=0\), and
\[
    \operatorname{Lip}(f|_{[0,r]})\longrightarrow0
    \qquad (r\downarrow0).
\]
\end{proposition}

\begin{proof}
Write \(\gamma'(0)=\lambda v\), where \(\lambda>0\), and set
\(s(t)=\langle\gamma(t)-x,v\rangle\). Since \(s'(0)=\lambda\), after
restriction the function \(s\) is strictly increasing and has a \(C^1\)
inverse \(t=\tau(s)\). Define
\(f(s):=\gamma(\tau(s))-x-sv\).
Then \(f(s)\in v^\perp\), and
\(f'(0)=\gamma'(0)/\lambda-v=0\). Continuity of \(f'\) gives
\[
    \operatorname{Lip}(f|_{[0,r]})
    \le
    \sup_{0\le s\le r}\|f'(s)\|
    \longrightarrow0.
\]
\end{proof}

\begin{corollary}[Regular \(C^1\) curve germs as local images of straight line germs]
\label{cor:c1-probes-straightening}
Up to restriction, regular \(C^1\) curve germs are precisely the local
\(C^1\)-diffeomorphic images of one-sided line germs.
\end{corollary}

\begin{proof}
Let \(\Psi\) be a local \(C^1\)-diffeomorphism with \(\Psi(0)=x\). After
restricting the positive line germ so that it lies in a neighborhood on
which \(\Psi\) is a diffeomorphism, its image is parametrized by
\(\gamma(s)=\Psi(sv)\). This parametrization is injective, and
\(\gamma'(s)=D\Psi(sv)v\ne0\), since \(D\Psi(sv)\) is invertible and
\(v\ne0\). Hence its image is a regular \(C^1\) curve germ.

Conversely, use the graph representation of
Proposition~\ref{prop:c1-vanishing-slope-normal-form} and extend \(f\) across
\(0\) by setting \(f(s)=0\) for \(s<0\). Since \(f(0)=f'(0)=0\), this extension
is \(C^1\) near \(0\). With respect to
\(\R^n=\R v\oplus v^\perp\), define, on a neighborhood of \(0\),
\[
    \Phi(sv+z):=x+sv+f(s)+z,
    \qquad
    D\Phi(0)(av+w)
    =av+a f'(0)+w
    =av+w,
\]
where \(z,w\in v^\perp\) and \(a\in\R\).
Thus \(D\Phi(0)\) is the identity. The inverse function theorem gives
neighborhoods on which \(\Phi\) is a \(C^1\)-diffeomorphism. After restricting
the line germ to such a neighborhood, the identity
\(\Phi(sv)=x+sv+f(s)=\gamma(s)\) shows that \(\Phi\) maps it onto the given
curve germ.
\end{proof}

Thus regular \(C^1\) curve germs form the smallest class of curve germs
containing one-sided line germs and stable under local
\(C^1\)-diffeomorphisms.

\begin{remark}[Local chords and base-point secants]
A preliminary version of the construction~\cite{pointcrossarxivv1} used
injective Lipschitz probes whose direction was prescribed only through the
limiting base-point secants
\(\bigl(\gamma(t)-x\bigr)/\|\gamma(t)-x\|\to v\). The present regular
\(C^1\)-probe formulation replaces that definition: the weaker condition
controls position relative to \(x\), but not the local directions of the trace,
and can therefore collect transverse complexity unrelated to the intended
first-order channel. Indeed, in the graph representation, for
\(0\le t<s\le r\),
\[
    \gamma(s)-\gamma(t)=(s-t)v+f(s)-f(t),
    \qquad
    \frac{\|f(s)-f(t)\|}{s-t}
    \le
    \operatorname{Lip}(f|_{[0,r]})
    \longrightarrow0
    \qquad(r\downarrow0).
\]
Thus all chords between points sufficiently close to the base point become
uniformly aligned with \(v\). By contrast, the limiting-secant condition
examines only chords having \(x\) as one endpoint and does not control chords
joining two nearby points with positive parameters.
Subsection~\ref{subsec:vs-c1-equivalence} shows that the corresponding
vanishing-slope trace class gives the same optimized weights, while
Subsection~\ref{subsec:lipschitz-secant-defect} shows that base-point secants
alone can collect transverse complexity excluded by regular \(C^1\) probes.
\end{remark}

\subsection{Realization of tangent orientations}
\label{subsec:c1-realization}

The realization of an arbitrary Bouligand orientation reduces to the following
interpolation lemma.

\begin{lemma}[\(C^1\) interpolation with vanishing slope]
\label{lem:c1-interpolation}
Let \(x\in\R^n\), let \(v\in S^{n-1}\), let
\(t_1>t_2>\cdots>0\) with \(t_j\to0\), and let \(z_j\in v^\perp\) satisfy
\(z_j\to0\) and
\[
    \frac{\|z_j-z_{j+1}\|}{t_j-t_{j+1}}
    \longrightarrow0.
\]
Then there exists \(f\in C^1([0,t_1],v^\perp)\) such that
\[
    f(0)=f'(0)=0,
    \qquad
    f(t_j)=z_j \quad (j\ge1).
\]
Consequently, \(\gamma(t)=x+tv+f(t)\) is injective and regular, with
\(\gamma'(0)=v\).
\end{lemma}

\begin{proof}
Let \(h(u)=3u^2-2u^3\) on \([0,1]\). For
\(t\in[t_{j+1},t_j]\), define
\[
    f(t)
    :=
    z_{j+1}
    +h\!\left(\frac{t-t_{j+1}}{t_j-t_{j+1}}\right)(z_j-z_{j+1}).
\]
The pieces agree at their endpoints and have zero derivative there. Moreover,
\[
    \sup_{t\in[t_{j+1},t_j]}\|f'(t)\|
    \le
    \|h'\|_\infty
    \frac{\|z_j-z_{j+1}\|}{t_j-t_{j+1}}
    \longrightarrow0.
\]
Set \(M_j:=\|h'\|_\infty\sup_{k\ge j}
\|z_k-z_{k+1}\|/(t_k-t_{k+1})\). Since the successive-slope ratios tend to
zero, \(M_j\to0\). Moreover, the preceding estimate gives
\(\|f'(t)\|\le M_j\) for \(0<t\le t_j\). Fix \(0<t\le t_j\), and choose \(m\)
so large that \(t_m<t\). Since \(f(t_m)=z_m\), piecewise integration yields
\[
    \|f(t)-z_m\|
    \le \int_{t_m}^{t}\|f'(u)\|\,du
    \le M_j(t-t_m).
\]
Letting \(m\to\infty\) and using \(t_m\to0\) and \(z_m\to0\), we obtain
\(\|f(t)\|\le M_jt\). Hence \(f(t)/t\to0\) as \(t\downarrow0\). Together with
\(f'(t)\to0\), this shows that setting \(f(0)=0\) and \(f'(0)=0\) gives a
\(C^1\) extension to \([0,t_1]\). Finally,
\(\langle\gamma(t)-x,v\rangle=t\), so distinct parameters have distinct
longitudinal coordinates. Moreover, \(f'(t)\perp v\), and hence
\(\|\gamma'(t)\|^2=1+\|f'(t)\|^2\ge1\).
\end{proof}

\begin{theorem}[Regular realization of Bouligand orientations]
\label{thm:c1-probe-realization}
Let \(A\subset\R^n\), let \(x\in\overline A\), and let \(v\in S^{n-1}\).
Then
\[
    v\in\Sigma_x(A)
    \quad\Longleftrightarrow\quad
    \mathcal G_x^A(v)\ne\varnothing.
\]
When \(v\in\Sigma_x(A)\), the probe may be chosen in the graph form
\(\gamma(t)=x+tv+f(t)\) of Lemma~\ref{lem:c1-interpolation} and may be made to
pass through a sequence of points of \(A\) converging to \(x\).
This is a recurrent realization: the probe is required to meet \(A\)
arbitrarily close to \(x\), but its punctured trace need not be contained in
\(A\).
\end{theorem}

\begin{proof}
If \(\gamma\in\mathcal G_x^A(v)\), recurrence gives \(s_j\downarrow0\) with
\(\gamma(s_j)\in A\). Since
\(\gamma(s)-x=s\gamma'(0)+o(s)\), the normalized secants converge to \(v\),
and hence \(v\in\Sigma_x(A)\).

Conversely, using the equivalent normalized-secant characterization of
\(\Sigma_x(A)\), choose \(y_j\in A\setminus\{x\}\) such that \(y_j\to x\) and
\((y_j-x)/\|y_j-x\|\to v\). Write
\[
    y_j-x=t_jv+z_j,
    \qquad
    t_j=\langle y_j-x,v\rangle,
    \qquad
    z_j\in v^\perp.
\]
Eventually \(t_j>0\), while \(t_j\to0\), \(\|z_j\|/t_j\to0\), and hence
\(z_j\to0\). Passing to a subsequence, assume \(t_{j+1}\le t_j/4\). If
\(\eta_j=\|z_j\|/t_j\), then
\[
    \frac{\|z_j-z_{j+1}\|}{t_j-t_{j+1}}
    \le
    \frac{4}{3}\eta_j+\frac{1}{3}\eta_{j+1}
    \longrightarrow0.
\]
Lemma~\ref{lem:c1-interpolation} produces a regular graph through every
\(y_j\). It is recurrent in \(A\) and belongs to \(\mathcal G_x^A(v)\).
\end{proof}

\begin{corollary}[Projective realization of tangent channels]
\label{cor:c1-projective-probe-realization}
Let \(A\subset\R^n\), let \(x\in\overline A\), and let
\(\xi\in\mathbb{RP}^{n-1}\). Then
\[
    \xi\in\mathbb P\Sigma_x(A)
    \quad\Longleftrightarrow\quad
    \mathcal G_x^A(\xi)\ne\varnothing.
\]
\end{corollary}

\begin{proof}
This follows by applying Theorem~\ref{thm:c1-probe-realization} to the two
unit representatives of \(\xi\).
\end{proof}

Equivalently, the union defining \(\mathcal G_x^A(\xi)\) may be restricted to
the representatives belonging to \(\Sigma_x(A)\): absent orientations
contribute empty families.

\begin{remark}[Sharpness within the H\"older scale]
The realization theorem cannot be strengthened uniformly to any Hölder class
\(C^{1,\alpha}\), \(0<\alpha\le1\). Indeed, fix \(0<\beta<\alpha\) and consider
\(A_\beta:=\{(t,t^{1+\beta}):0<t\le1\}\).
Then \(e_1\in\Sigma_0(A_\beta)\). A recurrent \(C^{1,\alpha}\) probe tangent
to \(e_1\) remains \(C^{1,\alpha}\) under longitudinal graph
reparametrization: the inverse derivative is the reciprocal of a nonvanishing
\(C^{0,\alpha}\) function. Its graph function \(f\) would satisfy
\(f(0)=f'(0)=0\), hence \(|f(t)|=O(t^{1+\alpha})\), which is incompatible
with \(f(t_j)=t_j^{1+\beta}\) along a sequence \(t_j\downarrow0\). Thus no
fixed \(C^{1,\alpha}\) class can replace \(C^1\) for arbitrary sets.
\end{remark}

\subsection{Local Point-Box complexity on traces}
\label{subsec:local-upper-box-probes}

Point-Box dimension measures multiscale covering complexity on trace
intersections, including disconnected or countable ones. We recall the
Euclidean facts used below. For a nonempty bounded set
\(E\subset\R^n\), let
\(N_\varepsilon(E)\) be the least number of sets of diameter at most
\(\varepsilon\) needed to cover \(E\), and set
\[
    \dimBox(E)
    :=
    \limsup_{\varepsilon\downarrow0}
    \frac{\log N_\varepsilon(E)}{\log(1/\varepsilon)}.
\]
We use the convention \(\dimBox(\varnothing)=0\).

\begin{definition}[Local upper box dimension]
\label{def:pbox-metric}
For \(A\subset\R^n\) and \(x\in\R^n\), define
\[
    \Pbox\{x\}_A
    :=
    \inf_{r>0}\dimBox(A\cap B(x,r)).
\]
Equivalently, this is the limit as \(r\downarrow0\). Its value is \(0\) when
\(x\notin\overline A\).
\end{definition}

This notation is inherited from the Point-Extended Box
framework~\cite{pointbox,pointcsf}. The following properties will be used without
further comment.

\begin{proposition}[Point-Box properties]
\label{prop:pbox-basic-package}
Let \(A,B\subset\R^n\), let \(x\in\R^n\), and let \(U\) be a neighborhood of
\(x\). Then
\[
    A\subset B
    \quad\Longrightarrow\quad
    \Pbox\{x\}_A\le\Pbox\{x\}_B,
\]
and
\[
    \Pbox\{x\}_A
    =
    \Pbox\{x\}_{A\cap U}
    =
    \Pbox\{x\}_{\overline A}.
\]
Moreover,
\[
    \Pbox\{x\}_{A\cup B}
    =
    \max\{\Pbox\{x\}_A,\Pbox\{x\}_B\}.
\]
If \(\Phi:U\to V\) is bi-Lipschitz between neighborhoods of \(x\) and
\(y=\Phi(x)\), then
\[
    \Pbox\{y\}_{\Phi(A\cap U)}=\Pbox\{x\}_A.
\]
More generally, for subsets \(E\subset\R^n\), \(F\subset\R^m\) and
\(y\in\R^m\), a bi-Lipschitz bijection
\(T:E\cup\{x\}\to F\cup\{y\}\) with \(T(x)=y\) gives
\(\Pbox\{x\}_E=\Pbox\{y\}_F\).
\end{proposition}

\begin{proof}
Monotonicity and locality are immediate. For \(0<r<R\),
\[
    \overline A\cap B(x,r)
    \subset
    \overline{A\cap B(x,R)}.
\]
Together with monotonicity and closure invariance of upper box dimension, this
gives \(\Pbox\{x\}_A=\Pbox\{x\}_{\overline A}\) after taking infima. The
inequalities
\[
    \max\{N_\varepsilon(E),N_\varepsilon(F)\}
    \le N_\varepsilon(E\cup F)
    \le N_\varepsilon(E)+N_\varepsilon(F)
\]
give the finite-union formula after taking logarithms, upper limits, and then
the limit as \(r\downarrow0\). For the final assertions, let \(L\ge1\) be a
common Lipschitz bound for a pointed map and its inverse. The image of the
domain intersected with \(B(x,r/L)\) lies in \(B(y,r)\), whereas every
image point in \(B(y,r)\) has its preimage in \(B(x,Lr)\). Comparing
coverings with scales multiplied by \(L\), then letting \(r\downarrow0\),
gives equality of local dimensions. Adjoining the base point does not affect
the dimension, and locality gives the neighborhood formulation.
\end{proof}

\begin{proposition}[Point-Box complexity on a regular probe]
\label{prop:pbox-c1-probe}
Let \(\gamma:[0,\delta]\to\R^n\) be an injective regular \(C^1\) map, with
\(x=\gamma(0)\). Then
\[
    \Pbox\{x\}_E\le1
    \quad\text{for every }E\subset\Gamma_\gamma,
    \qquad
    \Pbox\{x\}_{\Gamma_\gamma}=1.
\]
\end{proposition}

\begin{proof}
The map \(\gamma\) is Lipschitz on its compact parameter interval, so its trace
has upper box dimension at most one. Conversely, every neighborhood of \(x\)
contains a nondegenerate initial subarc. In the graph representation of
Proposition~\ref{prop:c1-vanishing-slope-normal-form}, projection onto
\(\R v\) maps this subarc onto a nondegenerate interval. Since a Lipschitz map
cannot increase upper box dimension, the subarc has upper box dimension at
least one. Thus every local intersection of \(\Gamma_\gamma\) with a
neighborhood of \(x\) has upper box dimension at least one.
\end{proof}

For a fixed set \(A\subset\R^n\), the function
\(x\mapsto\Pbox\{x\}_A\) is upper semicontinuous, as follows directly
from the defining infimum and the inclusion
\(B(y,r/2)\subset B(x,r)\) whenever \(\|y-x\|<r/2\).

\subsection{Separated witnesses and sequential realization}
\label{subsec:separated-witnesses}

Local Point-Box complexity can be retained by finite separated blocks in
successively smaller neighborhoods. Their union gives a sequence with one
accumulation point and the same local complexity. Under suitable slope
conditions, interpolation will then give a regular \(C^1\) curve through
the selected points.

For a bounded set \(E\), let \(S_\varepsilon(E)\) denote the largest
cardinality of an \(\varepsilon\)-separated subset of \(E\), where distinct
points have distance greater than \(\varepsilon\). The standard inequalities
\(N_{2\varepsilon}(E)\le S_\varepsilon(E)\le N_\varepsilon(E)\) show that
covering and separated-set formulations give the same upper box dimension.

\begin{lemma}[Separated witnesses below a local Point-Box exponent]
\label{lem:separated-witnesses-pbox}
Let \(E\subset\R^n\), let \(x\in\R^n\), and let
\(0\le s<\Pbox\{x\}_E\). For every \(r>0\), \(L>1\), and
\(\varepsilon_0>0\), there exist \(0<\varepsilon<\varepsilon_0\) and a finite
set \(F\subset(E\cap B(x,r))\setminus\{x\}\) such that \(F\) is
\(L\varepsilon\)-separated and \(\#F\ge\varepsilon^{-s}\).
Conversely, suppose that for some \(s\ge0\) and \(L>1\), for every \(r>0\)
and every \(\varepsilon_0>0\), there exist \(0<\varepsilon<\varepsilon_0\)
and an \(L\varepsilon\)-separated finite set
\[
    F\subset(E\cap B(x,r))\setminus\{x\},
    \qquad
    \#F\ge\varepsilon^{-s}.
\]
Then \(\Pbox\{x\}_E\ge s\).
\end{lemma}

\begin{proof}
Fix \(r>0\) and set \(K:=E\cap B(x,r)\). Choose
\(s'\) with \(s<s'<\dimBox(K)\). Such a choice is possible because
\(\dimBox(K)\ge\Pbox\{x\}_E>s\). Since separated sets give the same upper box
dimension as coverings, there are arbitrarily small \(\rho>0\) for which
\(S_\rho(K)\ge\rho^{-s'}\). Let \(P\subset K\) be a
\(\rho\)-separated set of this cardinality and put
\(\varepsilon:=\rho/L\). After removing \(x\), if it belongs to \(P\), the
remaining set \(F\) is \(L\varepsilon\)-separated and satisfies
\[
    \#F\ge L^{-s'}\varepsilon^{-s'}-1
    \ge \varepsilon^{-s}
\]
when \(\varepsilon\) is sufficiently small. The scale may also be chosen below
\(\varepsilon_0\).

Conversely, fix \(r>0\). A set of diameter at most \(\varepsilon\) contains at
most one point of an \(L\varepsilon\)-separated set. Hence
\(N_\varepsilon(E\cap B(x,r))\ge\varepsilon^{-s}\) along arbitrarily small
scales. Thus \(\dimBox(E\cap B(x,r))\ge s\), and the conclusion follows by
letting \(r\downarrow0\).
\end{proof}

\begin{lemma}[Sequential realization of local Point-Box complexity]
\label{lem:sequential-realization-pbox}
Let \(E\subset\R^n\) and let
\(x\in\overline{E\setminus\{x\}}\). Then there exists a countable set
\[
    F=\{p_j:j\ge1\}\subset E\setminus\{x\},
    \qquad
    p_j\longrightarrow x,
\]
such that \(x\) is the only accumulation point of \(F\) and
\(\Pbox\{x\}_F=\Pbox\{x\}_E\).
\end{lemma}

\begin{proof}
Set \(d:=\Pbox\{x\}_E\). If \(d=0\), choose any sequence of distinct points
of \(E\setminus\{x\}\) converging to \(x\). Its range \(F\) has \(x\) as its
only accumulation point, and monotonicity gives
\(0\le\Pbox\{x\}_F\le\Pbox\{x\}_E=0\).

Assume \(d>0\), and choose \(0<s_1<s_2<\cdots<d\) with \(s_m\uparrow d\).
We construct radii \(r_m\downarrow0\), scales \(\varepsilon_m\downarrow0\),
and finite \(2\varepsilon_m\)-separated sets
\[
    F_m\subset (E\cap B(x,r_m))\setminus\{x\},
    \qquad
    \#F_m\ge\varepsilon_m^{-s_m}.
\]
Choose \(r_1>0\). Given \(r_m\),
Lemma~\ref{lem:separated-witnesses-pbox}, with \(L=2\), supplies
\(F_m\) at an arbitrarily small scale \(\varepsilon_m\). Require
\(\varepsilon_1<1\) and, for \(m\ge2\),
\(\varepsilon_m<\min\{1/m,\varepsilon_{m-1}/2\}\). After selecting \(F_m\),
choose
\(0<r_{m+1}<\frac12\min\{r_m,\operatorname{dist}(x,F_m)\}\).

Let \(F:=\bigcup_{m\ge1}F_m\). Outside every neighborhood of \(x\), only
finitely many blocks occur, and each of them is finite. The blocks are
nonempty and disjoint, so \(x\) is the only accumulation point of \(F\),
and its points can be enumerated as a sequence converging to \(x\).

Fix \(r>0\). For every sufficiently large \(m\), one has
\(F_m\subset F\cap B(x,r)\), and hence
\(N_{\varepsilon_m}(F\cap B(x,r))\ge\#F_m
\ge\varepsilon_m^{-s_m}\). Taking the upper limit along these scales and
using \(s_m\uparrow d\) shows that the upper box dimension of
\(F\cap B(x,r)\) is at least \(d\). Since \(r>0\) was arbitrary,
\(\Pbox\{x\}_F\ge d\). The reverse inequality follows from \(F\subset E\)
and monotonicity.
\end{proof}

\begin{remark}[Pointwise and sequential localization]
The local quantity comes from~\cite{pointbox,pointcsf}. Global sequence
realization for upper box dimension appears
in~\cite{Ivanov2023,FalconerLiu2026}. The finite-block selection here uses
the same mechanism, localized at a prescribed point and its local exponent.
An arbitrary subsequence need not retain that exponent.

\end{remark}

\begin{remark}[Why Point-Box dimension]
These properties are precisely those needed for channel weights.
Closure invariance of \(\Pbox\) does not by itself give closure invariance of
the optimized profile: for a fixed probe trace \(\Gamma\), the sets
\(\overline{A\cap\Gamma}\) and \(\overline A\cap\Gamma\) need not coincide.
The complexity of the latter intersection may therefore have to be carried
by a different probe recurrent in \(A\).
Section~\ref{sec:closure-shadowing} provides this reconstruction.
\end{remark}

\section{Directional weights and the Point-Cross Dimension}
\label{sec:point-cross-dimension}

For each tangent channel, we measure the local upper box dimension of
intersections with regular \(C^1\) curves in that channel. We optimize over
these curves and sum the resulting weights over independent channels.

\subsection{The Point-Cross Dimension}
\label{subsec:point-cross-definition}

\begin{definition}[Directional weight]
\label{def:directional-contribution}
Let \(A\subset\R^n\), let \(x\in\overline A\), and let
\(\xi\in\mathbb P\Sigma_x(A)\). Define
\[
    \theta_x^A(\xi)
    :=
    \sup_{\gamma\in\mathcal G_x^A(\xi)}
    \Pbox\{x\}_{A\cap\Gamma_\gamma}.
\]
The function \(\theta_x^A:\mathbb P\Sigma_x(A)\to[0,1]\) is called the
\emph{weighted projective tangent profile} of \(A\) at \(x\).
\end{definition}

Corollary~\ref{cor:c1-projective-probe-realization} makes the probe family
nonempty. If both orientations occur, the weight takes the larger oriented
supremum. No symmetry is assumed. Attainment will be established in
Theorem~\ref{thm:directional-weight-attainment}.

Because the covering scale tends to zero inside each fixed neighborhood,
nondegenerate subarcs occurring in every neighborhood force full weight even
if their diameters vanish rapidly. Trace intersections may also be totally
disconnected or countable.

\begin{proposition}[Basic properties of directional weights]
\label{prop:directional-weight-basic-properties}
Let \(A,B\subset\R^n\) and \(x\in\R^n\). Then:
\begin{enumerate}
    \item For every \(\xi\in\mathbb P\Sigma_x(A)\),
    \(0\le\theta_x^A(\xi)\le1\).
    \item If \(A\subset B\), then
    \(\theta_x^A(\xi)\le\theta_x^B(\xi)\) on
    \(\mathbb P\Sigma_x(A)\). If \(U\) is a neighborhood of \(x\), then
    \(\theta_x^A=\theta_x^{A\cap U}\) on their common tangent support.
    \item If \(\xi\in\mathbb P\Sigma_x(A)\) and some
    \(\gamma\in\mathcal G_x^A(\xi)\) satisfies
    \(\gamma((0,\delta_0])\subset A\) for a \(\delta_0>0\), then
    \(\theta_x^A(\xi)=1\).
\end{enumerate}
\end{proposition}

\begin{proof}
The first assertion follows from
Proposition~\ref{prop:pbox-c1-probe}. If \(A\subset B\), every probe recurrent
in \(A\) is recurrent in \(B\), and Point-Box monotonicity applies to each
trace intersection. For locality, restrict each probe to an initial interval
whose trace lies in \(U\). The discarded compact tail stays a positive distance
from \(x\). The full-trace criterion follows from closure invariance of
\(\Pbox\) and \(\Pbox\{x\}_{\Gamma_\gamma}=1\).
\end{proof}

\begin{definition}[Point-Cross Dimension]
\label{def:point-cross-dimension}
For \(A\subset\R^n\) and \(x\in\R^n\), define
\[
    \dimcross\{x\}_A
    :=
    \sup_{\Xi\in\mathcal I_x(A)}
    \sum_{\xi\in\Xi}\theta_x^A(\xi).
\]
The empty family contributes \(0\). In particular,
\(\dimcross\{x\}_A=0\) when \(x\notin\overline A\).
\end{definition}

For \(\Xi=\{\xi_1,\ldots,\xi_k\}\in\mathcal I_x(A)\) with \(k\ge1\),
finiteness yields
\[
    \sup_{\substack{\gamma_i\in\mathcal G_x^A(\xi_i)\\1\le i\le k}}
    \sum_{i=1}^k
    \Pbox\{x\}_{A\cap\Gamma_{\gamma_i}}
    =
    \sum_{i=1}^k\theta_x^A(\xi_i).
\]
Each supremum can be approached independently, so choosing the tuple jointly
is equivalent to optimizing each channel separately. This imposes no condition
on whether the set contains combinations of points from different probes.

The quantities \(\Ptan\) and \(\dimcross\) optimize over the same independent
channel families. The former assigns unit weight to every channel, whereas the
latter uses the trace-complexity profile \(\theta_x^A\).

\begin{proposition}[Hierarchy and locality]
\label{prop:point-cross-basic-properties}
Let \(A,B\subset\R^n\) and \(x\in\R^n\). Then
\[
    0\le\Pvec\{x\}_A
    \le\dimcross\{x\}_A
    \le\Ptan\{x\}_A
    \le n.
\]
If \(A\subset B\), then
\(\dimcross\{x\}_A\le\dimcross\{x\}_B\). If \(U\) is a neighborhood of
\(x\), then \(\dimcross\{x\}_A=\dimcross\{x\}_{A\cap U}\). Thus
\(\dimcross\) depends only on the germ of \(A\) at \(x\).
\end{proposition}

\begin{proof}
For the lower inequality, choose a basis of
\(\Span(\operatorname{Dir}_A(x))\) from \(\operatorname{Dir}_A(x)\). The
associated straight probes have full punctured traces in \(A\), so the
corresponding projective channels all have weight one. For the upper
inequality, every channel weight is at most one, and every family in
\(\mathcal I_x(A)\) has cardinality at most \(\Ptan\{x\}_A\).

If \(A\subset B\), tangent-support monotonicity gives
\(\mathcal I_x(A)\subset\mathcal I_x(B)\), while
Proposition~\ref{prop:directional-weight-basic-properties} compares the
corresponding weights. Taking suprema proves monotonicity. For a neighborhood
\(U\) of \(x\), tangent-support locality gives
\(\mathcal I_x(A)=\mathcal I_x(A\cap U)\), and the same proposition identifies
the weights on this common support. This proves locality.
\end{proof}

\subsection{Ranks at complexity thresholds}
\label{subsec:threshold-ranks}

The sum over independent channels can be recovered by counting how many
independent directions remain above each weight threshold. This formulation
will also organize the union and product arguments.

For \(E\subset\R^n\), \(x\in\R^n\), and
\(\xi\in\mathbb{RP}^{n-1}\), extend the profile by zero:
\[
    \widehat{\theta}_x^E(\xi):=
    \begin{cases}
        \theta_x^E(\xi),&\xi\in\mathbb P\Sigma_x(E),\\
        0,&\xi\notin\mathbb P\Sigma_x(E).
    \end{cases}
\]

For \(0<t\le1\), define
\[
 V_x^A(t):=\Span\left(\bigcup_{\widehat\theta_x^A(\xi)\ge t}\xi\right),
 \qquad r_x^A(t):=\dim V_x^A(t).
\]
These subspaces decrease as \(t\) increases, so \(r_x^A\) is a bounded
decreasing function. Empty superlevel sets generate the zero subspace.

\begin{proposition}[Integral representation by threshold ranks]
\label{prop:threshold-integral}
For every \(A\subset\R^n\) and \(x\in\R^n\),
\[
 \dimcross\{x\}_A=\int_0^1 r_x^A(t)\,dt.
\]
\end{proposition}

\begin{proof}
Write \(\theta=\widehat\theta_x^A\) and \(r=r_x^A\). For every finite
independent channel family \(\Xi\),
\[
 \sum_{\xi\in\Xi}\theta(\xi)
 =\int_0^1\#\{\xi\in\Xi:\theta(\xi)\ge t\}\,dt
 \le\int_0^1r(t)\,dt.
\]
Taking the supremum proves the upper bound.

Fix an integer \(N\ge1\). Choose a basis of \(V_x^A(1)\) among
representatives of channels of weight at least one. Descending through
\(j=N-1,\ldots,1\), extend the current basis to one of \(V_x^A(j/N)\)
using representatives of channels of weight at least \(j/N\). Such an
extension exists because these representatives generate that subspace.
The final independent channel family \(\Xi_N\) contains at least
\(r(j/N)\) members of weight at least \(j/N\) for every \(j\). Hence
\[
 \dimcross\{x\}_A\ge\sum_{\xi\in\Xi_N}\theta(\xi)
 \ge\frac1N\sum_{j=1}^N
       \#\{\xi\in\Xi_N:\theta(\xi)\ge j/N\}
 \ge\frac1N\sum_{j=1}^N r(j/N).
\]
The right-hand sums converge to the integral because \(r\) is bounded
and monotone. This gives the reverse inequality.
\end{proof}

The proof is a finite-rank version of the weighted independence argument
associated with the greedy algorithm~\cite{Edmonds1971}. It requires no
continuity of the weights and no maximizing independent family.

\subsection{A first separating example: the fractional Cantor cross}
\label{subsec:fractional-cantor-cross}

The Cantor cross gives a first strict separation and shows that regular probes
may carry fractional weights on nonsmooth trace intersections.

\begin{example}[Cantor cross]
\label{ex:fractional-cantor-cross}
Let \(K\subset[0,1]\) be the middle-third Cantor set, let
\(d=\log 2/\log 3\), and set \(A:=K e_1\cup K e_2\subset\R^2\). Then, at the
origin,
\[
    \theta_0^A([e_1])=\theta_0^A([e_2])=d
\]
and
\[
    \Pvec\{0\}_A=0
    <\Pbox\{0\}_A=d
    <\dimcross\{0\}_A=2d
    <\Ptan\{0\}_A=2.
\]
\end{example}

\begin{proof}
The identity \(K\cap[0,3^{-m}]=3^{-m}K\) shows that
\(\Pbox\{0\}_K=d\). The finite-union formula therefore gives
\(\Pbox\{0\}_A=d\). Since \(K\) contains no interval,
\(\Pvec\{0\}_A=0\), while
\(\mathbb P\Sigma_0(A)=\{[e_1],[e_2]\}\), so
\(\Ptan\{0\}_A=2\).

The straight probe \(t\mapsto te_1\) gives
\(\theta_0^A([e_1])\ge d\). Conversely, let \(\gamma\) be a regular probe in
the orientation \(e_1\). After restriction, its first coordinate is strictly
positive away from the origin. Its trace therefore meets \(K e_2\) only at the
origin, while its intersection with \(K e_1\) is a subset of \(K e_1\).
Point-Box monotonicity gives \(\theta_0^A([e_1])\le d\). The vertical channel
is identical by symmetry. The two channels are projectively independent, so
their weights add.
\end{proof}

The filled cusp of Section~\ref{sec:discriminating-examples} reverses the
inequality between \(\Pbox\) and \(\dimcross\), so no general ordering holds.

\subsection{Non-determination by Point-Box and tangential ranks}
\label{subsec:point-cross-non-determination}

The Cantor cross separates several values on a single germ. The next family
gives the stronger conclusion that the Point-Cross Dimension can vary over a
continuum while the local Point-Box dimension and both unweighted ranks remain
fixed.

\begin{example}[A continuum of Point-Cross values with fixed comparison data]
\label{ex:point-cross-non-determination}
For \(0<d<1\), set \(r_d:=2^{-1/d}\in(0,1/2)\), and let
\(K_d\subset[0,1]\) be the self-similar Cantor set generated by
\[
    t\longmapsto r_d t,
    \qquad
    t\longmapsto 1-r_d+r_d t.
\]
Set \(A_d:=[0,1]e_1\cup K_d e_2\). Then
\[
    \begin{aligned}
        \Pvec\{0\}_{A_d}&=1,\\
        \Pbox\{0\}_{A_d}&=1,\\
        \Ptan\{0\}_{A_d}&=2,
    \end{aligned}
\]
whereas
\[
    \dimcross\{0\}_{A_d}=1+d.
\]
The full projective tangent support is \(\{[e_1],[e_2]\}\), independently
of \(d\). Thus this support and the same comparison values are compatible
with every Point-Cross value in \((1,2)\).
\end{example}

\begin{proof}
The standard separated self-similar estimate gives
\[
    \dimBox(K_d)
    =\frac{\log 2}{\log(1/r_d)}
    =d.
\]
The identity \(K_d\cap[0,r_d^k]=r_d^kK_d\) gives the same local value at the
origin. Since \(A_d\) contains the horizontal segment, finite stability gives
\(\Pbox\{0\}_{A_d}=1\). The horizontal direction is its only exact segment
direction because \(K_d\) contains no interval, while its projective tangent
support is \(\{[e_1],[e_2]\}\). Thus the three stated comparison values
follow.

The horizontal channel has weight one. The vertical weight is \(d\) by
the same restriction argument as in Example~\ref{ex:fractional-cantor-cross}:
a probe tangent to the vertical arm cannot meet the horizontal arm near
zero except at zero. The straight vertical probe attains \(d\), and the
two independent weights sum to \(1+d\).

\end{proof}

\subsection{Vanishing-slope equivalence}
\label{subsec:vs-c1-equivalence}

The graph representation of
Proposition~\ref{prop:c1-vanishing-slope-normal-form} suggests a larger
geometric trace class. Individual traces in this class need not be \(C^1\),
but we will show that allowing them does not increase the optimized channel
weights. The metric reduction is supplied by
Lemma~\ref{lem:sequential-realization-pbox}. The geometric reconstruction
uses Lemma~\ref{lem:c1-interpolation}.

\begin{definition}[Vanishing-slope trace]
\label{def:vanishing-slope-trace}
Let \(A\subset\R^n\), let \(x\in\overline A\), and let \(v\in S^{n-1}\).
A \emph{vanishing-slope trace for \(A\) at \(x\) in the oriented direction
\(v\)} is a set
\[
    \Gamma_f:=\{x+tv+f(t):0\le t\le\delta\},
\]
where \(f:[0,\delta]\to v^\perp\) is Lipschitz, \(f(0)=0\), and
\[
    \operatorname{Lip}(f|_{[0,r]})\longrightarrow0
    \qquad (r\downarrow0).
\]
We also require
\(x\in\overline{(A\cap\Gamma_f)\setminus\{x\}}\). The family of these traces
is denoted by \(\mathcal V_x^A(v)\). For
\(\xi\in\mathbb P\Sigma_x(A)\), set
\[
    \mathcal V_x^A(\xi)
    :=
    \bigcup_{\substack{v\in\Sigma_x(A)\\ [v]=\xi}}
    \mathcal V_x^A(v).
\]
\end{definition}

The longitudinal coordinate is \(t\), so the canonical parametrization is
injective. Moreover,
\(\|f(t)\|\le t\operatorname{Lip}(f|_{[0,t]})\), and hence \(f(t)=o(t)\).
Thus its base-point secants converge to \(v\), and all local chords satisfy
\[
    \frac{\|f(s)-f(t)\|}{s-t}
    \le \operatorname{Lip}(f|_{[0,r]})
    \longrightarrow0
    \qquad (0\le t<s\le r,\ r\downarrow0).
\]
Proposition~\ref{prop:c1-vanishing-slope-normal-form} puts every regular probe
in this class after restriction and longitudinal reparametrization. Together
with Theorem~\ref{thm:c1-probe-realization}, this shows that
\(\mathcal V_x^A(\xi)\ne\varnothing\) for every
\(\xi\in\mathbb P\Sigma_x(A)\).

\begin{remark}[The trace classes are different]
The vanishing-slope condition does not impose \(C^1\) regularity away from the
base point. Work in \(\R^2\), with \(x=0\) and \(v=e_1\). For \(j\ge1\), let
\(I_j=(2^{-(j+1)},2^{-j}]\), set \(q(t)=(-1)^j/j\) on \(I_j\), and define
\[
    q(0):=0,
    \qquad
    f(t):=\left(\int_0^t q(s)\,ds\right)e_2
    \quad (0\le t\le1/2).
\]
Taking \(A:=\Gamma_f\), one has
\(\operatorname{Lip}(f|_{[0,2^{-N}]})\le1/N\), so the graph of \(f\) has
vanishing slope and belongs to \(\mathcal V_0^A(e_1)\). At every \(2^{-j}\),
\(j\ge2\), its one-sided tangent directions differ. Since these corners
accumulate at the origin, no initial restriction of the graph is a regular
\(C^1\) curve germ.
\end{remark}

\begin{definition}[Vanishing-slope directional weight]
\label{def:vanishing-slope-weight}
For \(A\subset\R^n\), \(x\in\overline A\), and
\(\xi\in\mathbb P\Sigma_x(A)\), define
\[
    \theta_x^{A,\mathrm{VS}}(\xi)
    :=
    \sup_{\Gamma_f\in\mathcal V_x^A(\xi)}
    \Pbox\{x\}_{A\cap\Gamma_f}.
\]
This weight belongs to \([0,1]\), since every vanishing-slope trace is the
Lipschitz image of a compact interval.
\end{definition}

\begin{theorem}[\(C^1\) reconstruction without loss of Point-Box complexity]
\label{thm:vs-to-c1-reconstruction}
Let \(A\subset\R^n\), let \(x\in\overline A\), let
\(v\in\Sigma_x(A)\), and let \(\Gamma_f\in\mathcal V_x^A(v)\). Then there
exists \(\widetilde\gamma\in\mathcal G_x^A(v)\) such that
\[
    \Pbox\{x\}_{A\cap\Gamma_{\widetilde\gamma}}
    \ge
    \Pbox\{x\}_{A\cap\Gamma_f}.
\]
More precisely, there exists a countable set
\(F\subset(A\cap\Gamma_f)\setminus\{x\}\) such that
\[
    \Pbox\{x\}_F=\Pbox\{x\}_{A\cap\Gamma_f},
    \qquad
    F\subset A\cap\Gamma_{\widetilde\gamma}.
\]
\end{theorem}

\begin{proof}
Set \(E:=A\cap\Gamma_f\). The recurrence condition allows us to apply
Lemma~\ref{lem:sequential-realization-pbox}. We obtain a countable set
\(F\subset E\setminus\{x\}\), with \(x\) as its only accumulation point,
such that \(\Pbox\{x\}_F=\Pbox\{x\}_E\).
Every point of \(F\) has a unique representation
\(p=x+tv+f(t)\) with \(t>0\). The corresponding parameter set has \(0\) as
its only accumulation point and has finite intersection with \([r,\delta]\)
for every \(r>0\). Otherwise compactness and injectivity would give an
accumulation point of \(F\) different from \(x\). It can therefore be
enumerated as \(t_1>t_2>\cdots\downarrow0\). Write
\[
    p_j=x+t_jv+z_j,
    \qquad
    z_j=f(t_j).
\]
Since \(f\) is continuous and \(t_j\to0\), one has \(z_j\to0\). The
vanishing-slope condition also gives
\[
    \frac{\|z_j-z_{j+1}\|}{t_j-t_{j+1}}
    \le
    \operatorname{Lip}(f|_{[0,t_j]})
    \longrightarrow0.
\]
Lemma~\ref{lem:c1-interpolation} therefore gives
\(g\in C^1([0,t_1],v^\perp)\) with \(g(0)=g'(0)=0\) and
\(g(t_j)=z_j\). The map
\(\widetilde\gamma(t):=x+tv+g(t)\) is injective and regular, with
\(\widetilde\gamma'(0)=v\), and \(\widetilde\gamma(t_j)=p_j\in A\). Since
\(p_j\to x\), it is recurrent in \(A\) and therefore belongs to
\(\mathcal G_x^A(v)\). Hence \(F\subset A\cap\Gamma_{\widetilde\gamma}\), and
monotonicity yields
\[
    \Pbox\{x\}_{A\cap\Gamma_{\widetilde\gamma}}
    \ge \Pbox\{x\}_F
    =\Pbox\{x\}_{A\cap\Gamma_f}.
\]
\end{proof}

\begin{corollary}[Equivalence of \(C^1\) and vanishing-slope weights]
\label{cor:c1-vs-weight-equivalence}
For every \(A\subset\R^n\), every \(x\in\overline A\), and every
\(\xi\in\mathbb P\Sigma_x(A)\),
\[
    \theta_x^A(\xi)=\theta_x^{A,\mathrm{VS}}(\xi).
\]
\end{corollary}

\begin{proof}
Every regular \(C^1\) probe has a vanishing-slope trace after restriction, and
locality preserves the Point-Box value of its trace intersection. Therefore
\(\theta_x^A(\xi)\le\theta_x^{A,\mathrm{VS}}(\xi)\). Conversely,
Theorem~\ref{thm:vs-to-c1-reconstruction} replaces each vanishing-slope trace
by a regular \(C^1\) probe in the same orientation without decreasing that
value. This gives the reverse inequality.
\end{proof}

\subsection{Why base-point secant control is too weak for Lipschitz probes}
\label{subsec:lipschitz-secant-defect}

Vanishing transverse slope does not change the weights, but control of limiting
base-point secants alone does: it constrains position relative to \(x\), not
local slopes. For this subsection only, let
\(\theta_x^{A,\mathrm{Lip}}(\xi)\) denote the supremum in
Definition~\ref{def:directional-contribution} obtained by replacing
\(\mathcal G_x^A(\xi)\) with injective Lipschitz maps
\(\gamma:[0,\delta]\to\R^n\) such that \(\gamma(0)=x\), the trace meets
\(A\setminus\{x\}\) arbitrarily close to \(x\), and
\[
    \frac{\gamma(t)-x}{\|\gamma(t)-x\|}\longrightarrow v
\]
for some \(v\in\Sigma_x(A)\) with \([v]=\xi\).

\begin{example}[Transverse complexity collected by a Lipschitz probe]
\label{ex:lipschitz-product-defect}
Let
\[
    E:=\{0\}\cup\{2^{-j}:j\ge1\},
    \qquad
    I:=[0,1],
    \qquad
    A:=E\times I,
\]
and let \(x=(0,0)\). Then \(\Pbox\{0\}_E=0\), the projective tangent
support is
\[
    \mathbb P\Sigma_x(A)
    =
    \{[(a,b)]\in\mathbb{RP}^1:a\ge0,\ b\ge0\},
\]
and the weighted profile is
\[
    \theta_x^A([e_2])=1,
    \qquad
    \theta_x^A(\xi)=0
    \quad
    \bigl(\xi\in\mathbb P\Sigma_x(A)\setminus\{[e_2]\}\bigr).
\]
Consequently,
\[
    \dimcross\{x\}_A=1,
    \qquad\text{while}\qquad
    \theta_x^{A,\mathrm{Lip}}([e_1])=1.
\]
\end{example}

\begin{proof}
For every fixed \(r>0\), one has
\(N_\varepsilon(E\cap[0,r])=O(\log(1/\varepsilon))\), so
\(\Pbox\{0\}_E=0\). Every secant direction of \(A\) lies in the closed first
quadrant. Conversely, let \(v=(a,b)\in S^1\) with \(a,b\ge0\). If \(a>0\),
then \((2^{-j},(b/a)2^{-j})\in A\) for all sufficiently large \(j\), and
these points have direction \(v\). The vertical direction is supplied by
\(\{0\}\times I\). This proves the stated tangent support.

The vertical segment is a full probe, so \(\theta_x^A([e_2])=1\). Now let
\(\gamma\) be a regular probe in any nonvertical channel. Its admissible
orientation has positive first coordinate. After restriction, its trace is a
Lipschitz graph over the horizontal coordinate. Since \(A=E\times I\), the
trace intersection \(A\cap\Gamma_\gamma\) is contained in the Lipschitz image
of a subset of \(E\). Hence
\(\Pbox\{x\}_{A\cap\Gamma_\gamma}=0\). The profile formula and
\(\dimcross\{x\}_A=1\) follow.

For the larger Lipschitz class, write \(a_j=2^{-j}\) and set
\(V_j:=\{a_j\}\times[0,a_j^2]\subset A\).
Join \((a_j,0)\) to \((a_{j+1},a_{j+1}^2)\) by a straight segment \(C_j\).
The concatenation of the segments \(V_j\) and \(C_j\), together with the
origin, is an injective rectifiable arc. Indeed, successive pieces lie in
strips with disjoint interiors and meet only at their prescribed endpoints,
while the total length is bounded by a constant multiple of
\(\sum_j(a_j+a_j^2)\). On \(V_j\), one has \(0\le y/x\le a_j\), whereas on
\(C_j\), one has \(x\ge a_{j+1}\) and
\(0\le y\le a_{j+1}^2\), hence \(0\le y/x\le a_{j+1}\). Thus the normalized
secants converge uniformly to \(e_1\). An arclength parametrization from the
origin is therefore admissible for the comparison class. Its trace intersection
with \(A\) contains a nondegenerate vertical segment in every neighborhood of
\(x\), and consequently has local Point-Box dimension one. Since every
Lipschitz trace has upper box dimension at most one,
\(\theta_x^{A,\mathrm{Lip}}([e_1])=1\).
\end{proof}

Thus a tangent channel can have weight zero for regular \(C^1\) probes and
weight one for Lipschitz probes with the same limiting base-point secants.
The latter probes can contain transverse segments arbitrarily close to the
base point. Vanishing transverse slope excludes these segments.

\section{Sequential reconstruction and profile stability}
\label{sec:closure-shadowing}

The metric witnesses of Subsection~\ref{subsec:separated-witnesses} admit two
geometric uses. Blocks selected on probes of converging directions can be
glued into one regular probe. This yields attainment of each directional
weight and upper semicontinuity of the profile. Alternatively, a sequential
witness can be perturbed by amounts negligible relative to its longitudinal
gaps and then interpolated without Point-Box loss. This second construction
gives closure shadowing and, in Section~\ref{sec:quantitative-accessibility},
the power-rate accessibility bound.

\subsection{Reconstruction and attainment of directional weights}
\label{subsec:diagonal-gluing}

We first extend the finite-block selection used in sequential realization
to a sequence of different probes. The blocks must be ordered so that the
chords joining consecutive blocks have asymptotically the prescribed direction.

\begin{lemma}[Reconstruction from probes with converging directions]
\label{lem:diagonal-probe-gluing}
Let \(A\subset\R^n\), let \(x\in\overline A\), and let
\(\gamma_j\in\mathcal G_x^A(v_j)\), where \(v_j\in S^{n-1}\) and
\(v_j\to v\). Set
\[
    b_j:=\Pbox\{x\}_{A\cap\Gamma_{\gamma_j}},
    \qquad D:=\limsup_{j\to\infty}b_j.
\]
There exists \(\gamma\in\mathcal G_x^A(v)\) such that
\(\Pbox\{x\}_{A\cap\Gamma_\gamma}\ge D\).
\end{lemma}

\begin{proof}
If \(D=0\), recurrence allows us to choose
\(p_j\in(A\cap\Gamma_{\gamma_j})\setminus\{x\}\) with
\(\|p_j-x\|<1/j\) and
\(\|(p_j-x)/\|p_j-x\|-v_j\|<1/j\). Hence \(v\in\Sigma_x(A)\), and
Theorem~\ref{thm:c1-probe-realization} gives a recurrent probe in direction
\(v\). Its complexity is nonnegative, as required.

Suppose \(D>0\). Pass to a subsequence with \(b_j>0\) and \(b_j\to D\), and choose
\(0<s_j<b_j\) with \(s_j\to D\). After discarding finitely many terms,
\(\langle v_j,v\rangle>0\). Projection onto \(v\) therefore provides a
regular parameter near the base point of each probe. In this parameter,
after restriction,
\[
    \gamma_j(t)=x+tv+f_j(t),\qquad 0\le t\le\delta_j,
    \qquad f_j(t)\in v^\perp,
\]
where
\[
    f_j(0)=0,\qquad
    f_j'(0)=\frac{v_j}{\langle v_j,v\rangle}-v\longrightarrow0.
\]
Set \(\eta_j:=\|f_j'(0)\|+1/j\). Continuity of each derivative allows a
further restriction on which \(\sup_t\|f_j'(t)\|\le\eta_j\), and hence
\(\operatorname{Lip}(f_j)\le\eta_j\to0\). The restriction may depend on
\(j\). No common parameter interval is required.
The discarded compact tails stay away from \(x\), so these restrictions
do not change \(b_j\).

Choose blocks recursively. Take \(R_j<\min\{2^{-j},\delta_j/2\}\), and,
when \(j>1\), also require
\[
    R_j<\frac14\min_{p\in F_{j-1}}\langle p-x,v\rangle.
\]
Lemma~\ref{lem:separated-witnesses-pbox} supplies a finite
\(8\varepsilon_j\)-separated set
\[
    F_j\subset(A\cap\Gamma_{\gamma_j}\cap B(x,R_j))\setminus\{x\},
    \qquad \#F_j\ge\varepsilon_j^{-s_j},
\]
with \(\varepsilon_j<2^{-j}\) and, for \(j>1\),
\(\varepsilon_j<\varepsilon_{j-1}/2\).
For \(p\in F_j\), its longitudinal parameter satisfies
\(0<t=\langle p-x,v\rangle\le\|p-x\|<R_j\).
Thus, for \(j>1\), all parameters in block \(j\) are smaller than one
quarter of the smallest parameter in block \(j-1\).

Order each block by decreasing parameter and concatenate the blocks. Write
the resulting points as \(p_i=x+t_iv+z_i\), where \(t_i\downarrow0\)
strictly. For a point in block \(j\), the estimate
\(\|z_i\|=\|f_j(t_i)\|\le\eta_jt_i\) gives \(z_i\to0\).
Within block \(j\), consecutive transverse
slopes are at most \(\eta_j\). At a junction between blocks \(j\) and
\(j+1\), the parameters \(t>u\) satisfy \(u<t/4\), and the transverse
coordinates satisfy \(\|z\|\le\eta_jt\), \(\|w\|\le\eta_{j+1}u\).
Consequently,
\[
    \frac{\|z-w\|}{t-u}
    \le\frac{\eta_jt+\eta_{j+1}u}{t-u}
    \le\frac43\eta_j+\frac13\eta_{j+1}\longrightarrow0.
\]
Since every block is finite, all successive transverse slopes tend to zero
in this enumeration. Lemma~\ref{lem:c1-interpolation} gives a regular
\(C^1\) graph of direction \(v\) through all \(p_i\). The corresponding
probe \(\gamma\) is recurrent in \(A\).

For every fixed \(r>0\), all sufficiently late blocks lie in \(B(x,r)\)
and on \(\Gamma_\gamma\), so
\[
    N_{\varepsilon_j}(A\cap\Gamma_\gamma\cap B(x,r))
    \ge\#F_j\ge\varepsilon_j^{-s_j}.
\]
Taking the upper limit along \(\varepsilon_j\downarrow0\), and then the
infimum over \(r\), proves the lower bound \(D\).
\end{proof}

The construction preserves the counting information block by block. It does
not require the original probes to converge in a function-space topology.

\begin{theorem}[Attainment of directional weights]
\label{thm:directional-weight-attainment}
For every \(A\subset\R^n\), \(x\in\overline A\), and
\(\xi\in\mathbb P\Sigma_x(A)\), there exists
\(\gamma\in\mathcal G_x^A(\xi)\) such that
\[
    \Pbox\{x\}_{A\cap\Gamma_\gamma}=\theta_x^A(\xi).
\]
Thus the supremum in Definition~\ref{def:directional-contribution} is a
maximum.
\end{theorem}

\begin{proof}
If the weight is zero, any recurrent probe in the channel attains it.
Such a probe exists by Corollary~\ref{cor:c1-projective-probe-realization}.
Otherwise choose probes whose complexities tend to \(\theta_x^A(\xi)\).
Since the channel has only two orientations, a subsequence has one fixed
oriented direction \(v\). Lemma~\ref{lem:diagonal-probe-gluing} gives a
probe in direction \(v\) with complexity at least \(\theta_x^A(\xi)\).
The definition of the weight gives the reverse inequality.
\end{proof}

Consequently, for every fixed finite independent channel family, the tuple
optimization following Definition~\ref{def:point-cross-dimension} is attained
by choosing an optimizing regular probe in each channel. This imposes no
common-scale condition on the trace intersections and does not imply
attainment of the supremum over independent channel families.

\subsection{Upper semicontinuity and the scalar supremum}
\label{subsec:profile-upper-semicontinuity}

\begin{theorem}[Upper semicontinuity of the directional profile]
\label{thm:directional-profile-usc}
For every fixed \(A\subset\R^n\) and \(x\in\R^n\), the function
\(\widehat{\theta}_x^A\) is upper semicontinuous on
\(\mathbb{RP}^{n-1}\). In particular, the profile is Borel measurable,
and, when \(\mathbb P\Sigma_x(A)\ne\varnothing\), some channel has maximal
weight on this support.
\end{theorem}

\begin{proof}
Let \(\xi_j\to\xi\) and set
\(D:=\limsup_j\widehat{\theta}_x^A(\xi_j)\). If \(D=0\),
nonnegativity gives \(\widehat{\theta}_x^A(\xi)\ge D\).
Otherwise pass to a subsequence of positive-weight channels whose weights
tend to \(D\). By Theorem~\ref{thm:directional-weight-attainment}, choose
recurrent probes \(\gamma_j\) in those channels such that
\[
    \Pbox\{x\}_{A\cap\Gamma_{\gamma_j}}
    =\widehat{\theta}_x^A(\xi_j).
\]
Their complexities therefore tend to \(D\). By compactness of the sphere,
a subsequence of their oriented unit
directions converges to \(v\), with \([v]=\xi\).
Lemma~\ref{lem:diagonal-probe-gluing} produces a recurrent probe in
direction \(v\) with complexity at least \(D\). Thus \(\xi\) belongs to
the tangent support and \(\widehat{\theta}_x^A(\xi)\ge D\).
This proves upper semicontinuity and hence Borel measurability. For the
last assertion, the normalized-secant characterization gives
\[
    \Sigma_x(A)=\bigcap_{r>0}
    \overline{\left\{\frac{p-x}{\|p-x\|}:
            p\in(A\cap B(x,r))\setminus\{x\}\right\}}.
\]
Thus \(\Sigma_x(A)\) is closed in the sphere, and its projective image is
compact. Upper semicontinuity gives a maximal weight when this support
is nonempty.
\end{proof}

The preceding results concern individual channels. They do not turn the
scalar supremum over independent families into a maximum: independence is
not a closed condition. The following example makes this obstruction explicit.

\begin{example}[A continuous profile with no maximizing independent family]
\label{ex:scalar-nonattainment}
Let \(K_d\) be the Cantor set from
Subsection~\ref{subsec:point-cross-non-determination}, and put
\[
    d_j:=1-\frac1{j+1},\qquad
    v_j:=\frac{(1,1/j)}{\sqrt{1+j^{-2}}},\qquad
    A:=[0,1]e_1\cup\bigcup_{j\ge1}2^{-j}K_{d_j}v_j.
\]
Then \(A\subset\R^2\) is compact and
\[
    \mathbb P\Sigma_0(A)=\{[e_1]\}\cup\{[v_j]:j\ge1\},
    \qquad
    \theta_0^A([e_1])=1,\quad\theta_0^A([v_j])=d_j.
\]
The profile is continuous on its support, but
\(\dimcross\{0\}_A=2\) is not attained by any independent channel family.
\end{example}

\begin{proof}
Each arm is compact, and the diameters of the indexed arms tend to zero.
Any convergent sequence using unbounded arm indices has a subsequence tending
to \(0\in A\). The remaining limits lie in a finite union of compact arms.
Hence \(A\) is closed and bounded. Its secant directions at zero lie in the
displayed family, whose only limiting direction is already \([e_1]\).
Each arm has points arbitrarily close to zero, so every displayed channel
belongs to the tangent support.

The horizontal segment gives weight one. For a fixed \(j\), the channel
\([v_j]\) is isolated in the support. The secants of any recurrent probe
in that channel therefore force all its sufficiently small nonzero contacts
into the single arm \(2^{-j}K_{d_j}v_j\). Point-Box monotonicity gives
weight at most \(d_j\), while the straight probe gives equality, using the
local Cantor dimension computed in
Subsection~\ref{subsec:point-cross-non-determination}.

The independent pairs \(\{[e_1],[v_j]\}\) have total weight \(1+d_j\to2\),
so the ambient bound gives \(\dimcross\{0\}_A=2\). No pair attains two,
because only the horizontal channel has weight one, and an independent
family has at most two members. Finally, the channels \([v_j]\) are
isolated and \(d_j\to1=\theta_0^A([e_1])\), proving continuity on the
support. The pairs approaching the supremum lose independence in the limit.
\end{proof}

In Example~\ref{ex:scalar-nonattainment}, the threshold rank is two for
\(0<t<1\) and one for \(t=1\). Its integral equals two although no
independent family has total weight two. The threshold formula thus retains
the distinction between a supremum and a maximum.

\subsection{Controlled perturbation and regular reconstruction}
\label{subsec:controlled-perturbation}

For a sequential trace, perturbations small relative to neighboring
longitudinal gaps preserve both the metric complexity and the vanishing
transverse slopes. The next lemma isolates these two conclusions.

\begin{lemma}[Perturbative reconstruction without Point-Box loss]
\label{lem:perturbative-sequential-reconstruction}
Let \(x\in\R^n\), \(v\in S^{n-1}\), and
\[
    p_j=x+t_jv+z_j,\qquad t_j\downarrow0\text{ strictly},
    \qquad z_j\in v^\perp,\quad z_j\to0.
\]
Write \(g_j:=t_j-t_{j+1}\), and assume
\[
    \frac{\|z_j-z_{j+1}\|}{g_j}\longrightarrow0.
\]
Let \(e_j\in\R^n\) and \(q_j=p_j+e_j\), with
\begin{equation}
    \frac{\|e_j\|+\|e_{j+1}\|}{g_j}\longrightarrow0.
    \label{eq:perturbation-gap-control}
\end{equation}
For all sufficiently large \(N\), the map
\(T(x)=x\), \(T(p_j)=q_j\) for \(j\ge N\), is bi-Lipschitz between
\(F_N\cup\{x\}\) and \(Q_N\cup\{x\}\), where
\(F_N:=\{p_j:j\ge N\}\) and \(Q_N:=\{q_j:j\ge N\}\).
Moreover, there exists \(\gamma\in\mathcal G_x^{Q_N}(v)\) whose trace
contains \(Q_N\), and
\[
    \Pbox\{x\}_{Q_N}=\Pbox\{x\}_{F_N}
    =\Pbox\{x\}_{\{p_j:j\ge1\}}.
\]
\end{lemma}

\begin{proof}
Set
\[
    \eta_N:=\sup_{j\ge N}
       \frac{\|e_j\|+\|e_{j+1}\|}{g_j}.
\]
For large \(N\), this is finite and less than one, and \(\eta_N\to0\).
For \(N\le i<j\), telescoping gives
\[
    \|e_i-e_j\|
    \le\sum_{k=i}^{j-1}(\|e_k\|+\|e_{k+1}\|)
    \le\eta_N(t_i-t_j)
    \le\eta_N\|p_i-p_j\|.
\]
Also \(\|e_i\|\le\eta_Ng_i\le\eta_Nt_i\le\eta_N\|p_i-x\|\).
Consequently, for all \(p,p'\in F_N\cup\{x\}\),
\[
    (1-\eta_N)\|p-p'\|
    \le\|T(p)-T(p')\|
    \le(1+\eta_N)\|p-p'\|.
\]
In particular, the \(q_j\) are distinct, different from \(x\), and tend
to \(x\). Proposition~\ref{prop:pbox-basic-package}, applied to the
pointed bijection \(T\), gives
\(\Pbox\{x\}_{Q_N}=\Pbox\{x\}_{F_N}\).
The finitely many omitted points \(p_1,\ldots,p_{N-1}\) stay a positive
distance from \(x\), so locality gives the remaining equality. No ambient
extension of \(T\) is needed.

For reconstruction, decompose
\(e_j=a_jv+b_j\), with \(b_j\in v^\perp\), and set
\(\tau_j:=t_j+a_j\), \(w_j:=z_j+b_j\). Then
\[
    \tau_j\ge(1-\eta_N)t_j>0,\qquad
    \tau_j-\tau_{j+1}\ge(1-\eta_N)g_j>0
    \quad(j\ge N).
\]
We have \(\tau_j\to0\) and \(w_j\to0\), since
\(\|e_j\|\le\eta_Nt_j\to0\). In addition,
\[
    \frac{\|w_j-w_{j+1}\|}{\tau_j-\tau_{j+1}}
    \le
    \frac{\|z_j-z_{j+1}\|/g_j+
          (\|e_j\|+\|e_{j+1}\|)/g_j}
         {1-(\|e_j\|+\|e_{j+1}\|)/g_j}
    \longrightarrow0.
\]
Lemma~\ref{lem:c1-interpolation} therefore supplies a \(C^1\) graph
\(x+tv+\widetilde f(t)\) through \(q_j=x+\tau_jv+w_j\), \(j\ge N\),
with \(\widetilde f(0)=\widetilde f'(0)=0\). Its longitudinal coordinate
makes it injective and regular, and the sequence \(q_j\to x\) gives
recurrence.
\end{proof}

The gap condition already implies \(\|e_j\|/t_j\to0\). No separate
base-point error condition or blockwise recounting is needed once
\eqref{eq:perturbation-gap-control} has been verified.

\subsection{\texorpdfstring{\(C^1\)}{C1} shadowing of closure traces}
\label{subsec:c1-closure-shadowing}

Closure invariance of the local Point-Box dimension does not by itself imply
closure invariance of the directional weights. A probe may meet \(\overline A\)
with high local complexity while its intersection with \(A\) along the same
trace has much smaller complexity. The probe therefore cannot generally be
kept fixed.

For a concrete obstruction, let \(\operatorname{den}(q)\) denote the positive
denominator of \(q\in\mathbb Q\) in lowest terms and set
\[
    A:=\bigl\{(q,\operatorname{den}(q)^{-2}):
    q\in\mathbb Q\cap(0,1]\bigr\}\subset\R^2.
\]
Approximating each \(t\in[0,1]\) by rationals with denominators tending to
infinity shows that \([0,1]\times\{0\}\subset\overline A\). At the origin,
the straight probe \(\gamma(t)=(t,0)\) is therefore admissible for
\(\overline A\) and satisfies
\(\Pbox\{0\}_{\overline A\cap\Gamma_\gamma}=1\).
However, \(A\cap\Gamma_\gamma=\varnothing\), so the same probe is not
admissible for \(A\). Closure shadowing must generally change the trace, not
merely compare the two intersections on a fixed trace.

We now select a sequential witness on the closure trace, move its points
into \(A\), and apply Lemma~\ref{lem:perturbative-sequential-reconstruction}.

\begin{lemma}[\(C^1\) shadowing without Point-Box loss]
\label{lem:c1-closure-trace-shadowing}
Let \(A\subset\R^n\), let \(x\in\overline A\), let \(v\in S^{n-1}\), and let
\(\gamma\in\mathcal G_x^{\overline A}(v)\). Set
\(E:=\overline A\cap\Gamma_\gamma\). Then there exists
\(\widetilde\gamma\in\mathcal G_x^A(v)\) such that
\[
    \Pbox\{x\}_{A\cap\Gamma_{\widetilde\gamma}}
    \ge
    \Pbox\{x\}_E.
\]
\end{lemma}

\begin{proof}
By locality, we may restrict \(\gamma\) and replace \(E\) by the
corresponding trace intersection without changing \(\Pbox\{x\}_E\).
Reparametrization does not change the trace.
Proposition~\ref{prop:c1-vanishing-slope-normal-form} then gives
\[
    \gamma(t)=x+tv+f(t),
    \qquad
    f\in C^1([0,\delta],v^\perp),
    \qquad
    f(0)=f'(0)=0.
\]
In particular,
\(\operatorname{Lip}(f|_{[0,r]})\to0\) as \(r\downarrow0\).

Recurrence and Lemma~\ref{lem:sequential-realization-pbox} give a countable
set \(F\subset E\setminus\{x\}\) with only accumulation point \(x\) and
\(\Pbox\{x\}_F=\Pbox\{x\}_E\). Since the graph parameter is one-to-one
and its inverse is continuous, we can order its points as
\[
    p_j=x+t_jv+f(t_j),\qquad t_j\downarrow0\text{ strictly}.
\]
Set \(g_j:=t_j-t_{j+1}>0\). The transverse slopes satisfy
\[
    \frac{\|f(t_j)-f(t_{j+1})\|}{g_j}
    \le\operatorname{Lip}(f|_{[0,t_j]})\longrightarrow0.
\]
Choose \(\alpha_j>0\) with \(\alpha_j\to0\). Since
\(p_j\in\overline A\), choose \(q_j\in A\) such that
\[
    \|q_j-p_j\|<\alpha_j\min\{g_{j-1},g_j\},
\]
where \(g_0\) is omitted. Then
\[
    \frac{\|q_j-p_j\|+\|q_{j+1}-p_{j+1}\|}{g_j}
    \le\alpha_j+\alpha_{j+1}\longrightarrow0.
\]
Lemma~\ref{lem:perturbative-sequential-reconstruction} gives, for a
sufficiently large \(N\), a regular probe \(\widetilde\gamma\) in direction
\(v\) through \(Q_N:=\{q_j:j\ge N\}\), with
\[
    \Pbox\{x\}_{Q_N}=\Pbox\{x\}_F=\Pbox\{x\}_E.
\]
The points of \(Q_N\) are distinct, different from \(x\), and converge to
\(x\). Thus \(\widetilde\gamma\in\mathcal G_x^A(v)\), and monotonicity
applied to \(Q_N\subset A\cap\Gamma_{\widetilde\gamma}\) proves the claim.
\end{proof}

The same construction covers zero and positive complexity. It transfers a
sequential witness, not the whole closure trace, and uses neither attainment
of the weight nor closure invariance of the weighted profile.

\subsection{Closure invariance of the weighted profile}
\label{subsec:closure-invariance}

\begin{theorem}[Closure invariance of the weighted projective tangent profile]
\label{thm:theta-closure-invariance}
Let \(A\subset\R^n\) and let \(x\in\overline A\). Then
\[
    \mathbb P\Sigma_x(A)=\mathbb P\Sigma_x(\overline A),
    \qquad
    \theta_x^A(\xi)=\theta_x^{\overline A}(\xi)
\]
for every channel \(\xi\) in this common support.
\end{theorem}

\begin{proof}
Closure invariance of the projective tangent support follows from
Proposition~\ref{prop:exact-tangential-ranks}. Since
\(A\subset\overline A\), monotonicity of directional weights gives
\(\theta_x^A(\xi)\le\theta_x^{\overline A}(\xi)\).

For the reverse inequality, fix any
\(\gamma\in\mathcal G_x^{\overline A}(\xi)\) and choose its oriented
direction \(v\).
Lemma~\ref{lem:c1-closure-trace-shadowing} gives
\(\widetilde\gamma\in\mathcal G_x^A(v)\), and hence
\[
    \theta_x^A(\xi)
    \ge\Pbox\{x\}_{A\cap\Gamma_{\widetilde\gamma}}
    \ge\Pbox\{x\}_{\overline A\cap\Gamma_\gamma}.
\]
Taking the supremum over
\(\gamma\in\mathcal G_x^{\overline A}(\xi)\) gives
\(\theta_x^A(\xi)\ge\theta_x^{\overline A}(\xi)\).
\end{proof}

\begin{corollary}[Closure invariance of the Point-Cross Dimension]
\label{cor:point-cross-closure-invariance}
For every \(A\subset\R^n\) and every \(x\in\R^n\),
\[
    \dimcross\{x\}_A=\dimcross\{x\}_{\overline A}.
\]
\end{corollary}

\begin{proof}
If \(x\notin\overline A\), both sides vanish. Otherwise,
Theorem~\ref{thm:theta-closure-invariance} identifies the projective tangent
supports and all their weights. The independent families are therefore the
same, so the two defining suprema agree.
\end{proof}

\begin{remark}[Dense sets]
Closure invariance identifies the profile with an invariant of the closed
germ. In particular, \(\widehat\theta_x^{\mathbb Q^n}(\xi)=1\) for every
\(x\in\R^n\) and every projective channel \(\xi\), and
\(\dimcross\{x\}_{\mathbb Q^n}=n\). Indeed, the closure is \(\R^n\),
where every channel contains a segment. Countability therefore does not
prevent maximal weight. The weights record local covering complexity,
rather than Hausdorff dimension, of the optimized intersections.
\end{remark}

\section{\texorpdfstring{\(C^1\)}{C1}-covariance and saturation}
\label{sec:structural-properties}

Closure invariance of the weighted profile required a new probe reconstructed
from a perturbed sequential witness. The behavior under a local \(C^1\)-change of
coordinates is more direct: the derivative transports tangent channels,
composition transports regular \(C^1\) probes, and local bi-Lipschitz
invariance preserves the Point-Box complexity of their trace intersections.

Every channel has weight one in two classes. For \(C^1\)-submanifolds,
every tangent channel has full weight by local straightening. For locally
subanalytic germs, analytic curve selection produces an arc realizing each
prescribed Bouligand orientation. A radial reparametrization turns that arc
into a regular \(C^1\) probe when its original derivative vanishes at the base
point.

\subsection{Channelwise covariance}
\label{subsec:c1-invariance-point-cross}

Let \(U,V\subset\R^n\) be open, let
\(\Phi:U\to V\) be a \(C^1\)-diffeomorphism, let
\(x\in U\cap\overline A\), and set \(y:=\Phi(x)\),
\(B:=\Phi(A\cap U)\), and \(L:=D\Phi(x)\).
Apply Proposition~\ref{prop:ptan-c1-covariance} to \(A\cap U\). By locality,
its tangent support at \(x\) agrees with that of \(A\), and the projective
formulation in Subsection~\ref{subsec:projective-tangent-support} shows that
the derivative induces a bijection preserving projective independence,
\[
    \Phi_{x,\mathbb P}:
    \mathbb P\Sigma_x(A)\longrightarrow\mathbb P\Sigma_y(B),
    \qquad
    \Phi_{x,\mathbb P}([v]):=[Lv].
\]
The principal statement transports the weight attached to each individual
channel, not only the Point-Cross Dimension.

\begin{theorem}[Channelwise covariance under local \(C^1\)-diffeomorphisms]
\label{thm:c1-theta-covariance}
Under the preceding assumptions,
\[
    \theta_y^B\bigl(\Phi_{x,\mathbb P}(\xi)\bigr)
    =
    \theta_x^A(\xi)
    \qquad
    \text{for every }\xi\in\mathbb P\Sigma_x(A).
\]
\end{theorem}

\begin{proof}
Fix \(\xi\in\mathbb P\Sigma_x(A)\) and
\(\gamma\in\mathcal G_x^A(\xi)\). Choose its oriented direction
\(v\in\Sigma_x(A)\), so that \([v]=\xi\) and
\(\gamma'(0)=\lambda v\) for some \(\lambda>0\). After restricting to a
smaller initial interval, the trace of \(\gamma\) lies in a neighborhood
\(U_0\) of \(x\) on which \(\Phi\) is bi-Lipschitz. This restriction does not
change the local Point-Box dimension of the trace intersection.

Set \(\alpha:=\Phi\circ\gamma\) on the restricted interval. It is injective
and regular, and the chain rule gives
\(\alpha'(0)=D\Phi(x)\gamma'(0)=\lambda Lv\).
Recurrence is preserved, so
\(\alpha\in\mathcal G_y^B(Lv/\|Lv\|)\). Injectivity of \(\Phi\) gives
\[
    B\cap\Gamma_\alpha
    =
    \Phi\bigl((A\cap U)\cap\Gamma_\gamma\bigr)
    =
    \Phi\bigl(A\cap\Gamma_\gamma\bigr),
\]
because the restricted trace satisfies \(\Gamma_\gamma\subset U_0\subset U\).
Therefore local bi-Lipschitz invariance of \(\Pbox\) yields
\[
    \Pbox\{y\}_{B\cap\Gamma_\alpha}
    =
    \Pbox\{x\}_{A\cap\Gamma_\gamma}.
\]
Taking the supremum over source probes proves
\(\theta_x^A(\xi)\le
\theta_y^B(\Phi_{x,\mathbb P}(\xi))\). For the reverse inequality, apply the
preceding argument to \(\Phi^{-1}\) at \(y\). Here
\(D\Phi^{-1}(y)=L^{-1}\), \(\Phi^{-1}(B)=A\cap U\), and the projective map
induced by \(L^{-1}\) sends \(\Phi_{x,\mathbb P}(\xi)\) back to \(\xi\).
The same argument therefore gives
\[
    \theta_y^B\bigl(\Phi_{x,\mathbb P}(\xi)\bigr)
    \le
    \theta_x^{A\cap U}(\xi).
\]
Locality of directional weights gives
\(\theta_x^{A\cap U}(\xi)=\theta_x^A(\xi)\), proving the reverse inequality.
\end{proof}

\begin{corollary}[Local \(C^1\)-diffeomorphism invariance]
\label{cor:c1-point-cross-invariance}
Under the preceding assumptions,
\[
    \dimcross\{\Phi(x)\}_{\Phi(A\cap U)}
    =
    \dimcross\{x\}_A.
\]
\end{corollary}

\begin{proof}
The map \(\Phi_{x,\mathbb P}\) bijectively transports the independent channel
families and, by Theorem~\ref{thm:c1-theta-covariance}, preserves the weight of
each channel. The corresponding sums, and hence their suprema, are equal.
\end{proof}

\begin{remark}[Definition on a \(C^1\) manifold]
Let \(M\) be a \(C^1\) manifold and \(A\subset M\). In a chart at
\(x\), define the weighted tangent support of the chart image and transport
its channels to \(T_xM\) by the inverse chart derivative.
Theorem~\ref{thm:c1-theta-covariance} makes the result independent of the
chart. The threshold subspaces transform by the same derivative. The scalar
quantity, closure invariance, and the product formula therefore extend in
charts to subsets of \(C^1\) manifolds.
\end{remark}

\subsection{\texorpdfstring{\(C^1\)}{C1}-submanifold saturation}
\label{subsec:smooth-saturation}

On a submanifold, local straightening reduces every tangent channel to a line
contained in a coordinate plane.

\begin{theorem}[\(C^1\)-submanifold saturation]
\label{thm:smooth-saturation}
Let \(M\subset\R^n\) be a \(C^1\) embedded submanifold without boundary of
dimension \(m\), and let \(x\in M\). Then
\[
    \theta_x^M(\xi)=1
    \qquad
    \text{for every }\xi\in\mathbb P\Sigma_x(M),
\]
and consequently
\[
    \dimcross\{x\}_M=\Ptan\{x\}_M=m.
\]
\end{theorem}

\begin{proof}
The local submanifold normal
form~\cite[Chapter~1, Section~1]{Hirsch1976} gives a local
\(C^1\)-diffeomorphism sending the germ of \(M\) at \(x\) to the germ of
\(\R^m\times\{0\}^{n-m}\) at the origin. Every projective tangent channel of
this plane contains a one-sided line segment whose punctured trace lies in the
plane, so its weight is one by
Proposition~\ref{prop:directional-weight-basic-properties}. Channelwise
covariance transports these weights back to \(M\). The tangent support spans
an \(m\)-dimensional space, and the defining supremum therefore gives
\(\dimcross\{x\}_M=\Ptan\{x\}_M=m\).
\end{proof}

\subsection{Subanalytic saturation}
\label{subsec:subanalytic-saturation}

We say that \(A\) is locally subanalytic at \(x\) when its intersection with
some neighborhood of \(x\) is subanalytic there. Curve selection on the
spherical blow-up realizes each prescribed Bouligand orientation by an analytic
arc in \(A\). The analytic parametrization itself may have zero derivative at
the base point, so it need not be an admissible probe. The next lemma removes
this defect without changing the trace.

\begin{lemma}[Full \(C^1\) realization of subanalytic tangent orientations]
\label{lem:subanalytic-channel-realization}
Let \(A\subset\R^n\) be locally subanalytic at \(x\), let
\(x\in\overline A\), and let \(v\in\Sigma_x(A)\). Then there exist
\(\delta>0\) and \(\gamma\in\mathcal G_x^A(v)\) such that
\[
    \gamma((0,\delta])\subset A\setminus\{x\}.
\]
The trace of \(\gamma\) may be chosen to be the trace of a real-analytic arc
whose punctured trace lies in \(A\).
\end{lemma}

\begin{proof}
By locality, we may replace \(A\) by its intersection with an open
neighborhood \(N\) of \(x\) in which it is subanalytic. Choose
\(\varepsilon>0\) so that \(x+ru\in N\) whenever
\(|r|<\varepsilon\) and \(u\in S^{n-1}\), and set
\(\Omega:=(-\varepsilon,\varepsilon)\times S^{n-1}\).
Form the positive spherical lift, regarded as a subset of \(\Omega\),
\[
    \widetilde A_x
    :=
    \bigl\{(r,u)\in(0,\varepsilon)\times S^{n-1}:
    x+ru\in A\bigr\}.
\]
The map \(F:\Omega\to N\), \(F(r,u):=x+ru\), is analytic. Since subanalytic
sets are stable under analytic inverse images, \(F^{-1}(A)\) is
subanalytic~\cite[Proposition~2.7(vi)]{Pati2003}. Intersecting it with the semianalytic
set \(\{(r,u)\in\Omega:r>0\}\) shows that \(\widetilde A_x\) is subanalytic
in \(\Omega\), which contains \((0,v)\).

The normalized-secant characterization of \(\Sigma_x(A)\) shows that
\((0,v)\in\overline{\widetilde A_x}\). Subanalytic curve selection, in its
analytic-arc form, therefore gives real-analytic functions \(r(t)\) and
\(u(t)\), defined near \(0\), such that \(r(0)=0\) and \(u(0)=v\),
and, for all sufficiently small \(t>0\),
\(r(t)>0\), \(\|u(t)\|=1\), and
\(x+r(t)u(t)\in A\setminus\{x\}\). This is the analytic-arc form of
Hironaka's selection lemma~\cite[Section~2]{KoikePaunescu2009}. Since
\(r\) is analytic and positive for
small positive \(t\), there are an integer \(k\ge1\) and \(c>0\) such that
\[
    r(t)=ct^k+O(t^{k+1}),
    \qquad
    r'(t)=kct^{k-1}+O(t^k).
\]
Consequently,
\[
    \frac{r(t)}{r'(t)}
    =
    \frac{t}{k}+O(t^2)
    \qquad(t\downarrow0).
\]
After restriction, \(r'(t)>0\) for \(t>0\), so \(r\) is strictly increasing.
Choose \(t_0>0\), set \(s_0:=r(t_0)\), and denote the \(C^1\) inverse of
\(r|_{(0,t_0]}\) on \((0,s_0]\) by \(\tau\). Define
\[
    \gamma(0):=x,
    \qquad
    \gamma(s):=x+s\,u(\tau(s))
    \quad(0<s\le s_0).
\]
For \(s>0\),
\[
    \gamma'(s)
    =u(\tau(s))+s u'(\tau(s))\tau'(s),
    \qquad
    s\tau'(s)
    =\frac{r(\tau(s))}{r'(\tau(s))}
    =\frac{\tau(s)}{k}+O(\tau(s)^2)
    \longrightarrow0.
\]
Since \(u\) is real analytic near \(0\), its derivative is continuous and
hence bounded after a further restriction. Therefore
\(s u'(\tau(s))\tau'(s)\to0\),
and, since \(u(\tau(s))\to v\), the derivative formula gives
\(\gamma'(s)\to v\).

On \((0,s_0]\), both \(\tau\) and \(\gamma\) are \(C^1\). Moreover,
\((\gamma(s)-x)/s=u(\tau(s))\to v\).
Thus \(\gamma\) is differentiable at \(0\), with \(\gamma'(0)=v\). Since the
preceding derivative formula gives
\(\gamma'(s)\to v=\gamma'(0)\), the derivative extends continuously to
\(0\). Hence \(\gamma\in C^1([0,s_0],\R^n)\). Its derivative is nonzero after
a further restriction. Moreover, \(\|\gamma(s)-x\|=s\), so \(\gamma\) is
injective. Finally,
\[
    \gamma(s)
    =x+r(\tau(s))u(\tau(s))
    \in A\setminus\{x\}
    \qquad(0<s\le s_0).
\]
Its trace is the same as that of the selected analytic arc, and the required
full regular \(C^1\) probe follows.
\end{proof}

\begin{theorem}[Subanalytic saturation]
\label{thm:subanalytic-saturation}
Let \(A\subset\R^n\) be locally subanalytic at \(x\), with
\(x\in\overline A\). Then
\[
    \theta_x^A(\xi)=1
    \qquad
    \text{for every }\xi\in\mathbb P\Sigma_x(A),
\]
and consequently
\[
    \dimcross\{x\}_A
    =
    \Ptan\{x\}_A
    =
    \dim\Span\bigl(\Tan_x(A)\bigr).
\]
\end{theorem}

\begin{proof}
For \(\xi\in\mathbb P\Sigma_x(A)\), choose
\(v\in\Sigma_x(A)\) with \([v]=\xi\). Lemma~\ref{lem:subanalytic-channel-realization}
provides a probe in this orientation whose entire punctured trace lies in
\(A\). Proposition~\ref{prop:directional-weight-basic-properties} gives
\(\theta_x^A(\xi)=1\).

Let \(d:=\Ptan\{x\}_A\). The support contains a projectively independent
family of \(d\) channels, all with weight one, so
\(\dimcross\{x\}_A\ge d\). The reverse inequality is the hierarchy in
Proposition~\ref{prop:point-cross-basic-properties}.
\end{proof}

\begin{corollary}[Point-Cross obstructions to local subanalyticity]
\label{cor:point-cross-obstructions-subanalyticity}
Let \(A\subset\R^n\) and \(x\in\overline A\). The germ of \(A\) at \(x\) is
not subanalytic if \(\theta_x^A(\xi)<1\) for some
\(\xi\in\mathbb P\Sigma_x(A)\), if
\(\dimcross\{x\}_A<\Ptan\{x\}_A\), or if
\(\dimcross\{x\}_A\notin\{0,1,\ldots,n\}\).
\end{corollary}

\begin{proof}
Each condition contradicts Theorem~\ref{thm:subanalytic-saturation}, since
\(\Ptan\{x\}_A=\dim\Span(\Tan_x(A))\) is an integer between \(0\) and \(n\).
\end{proof}

Subanalytic saturation also marks a limitation of the invariant. Within the
locally subanalytic category, the weighted profile contains no information
beyond the projective tangent support: every available channel already has
maximal weight. Consequently, deficient or fractional channel weights can
occur only outside this class. The converse is not asserted: a
non-subanalytic germ may still have full weight on every tangent channel.

\section{Directional calculus}
\label{sec:directional-calculus}

Finite unions, intersections, and Cartesian products behave differently.
For finite unions, each channel takes the largest component weight. For
intersections, we obtain upper bounds, while Cartesian products satisfy exact
scalar additivity. Unions of fractal sets on independent lines or curves
illustrate the finite-union formula.

\subsection{Finite unions and fractal frames}
\label{subsec:finite-unions-fractal-frames}

We use the zero-extended profile \(\widehat{\theta}_x^E\) introduced in
Subsection~\ref{subsec:threshold-ranks}.

\begin{theorem}[Exact finite-union calculus]
\label{thm:exact-finite-union-calculus}
Let \(A_1,\ldots,A_m\subset\R^n\), with \(m\ge1\), and let
\(x\in\overline{\bigcup_{j=1}^m A_j}\).
Then
\[
    \mathbb P\Sigma_x
    \left(
        \bigcup_{j=1}^m A_j
    \right)
    =
    \bigcup_{j=1}^m\mathbb P\Sigma_x(A_j),
\]
and, for every \(\xi\in\mathbb{RP}^{n-1}\),
\[
    \widehat{\theta}_x^{\bigcup_{j=1}^m A_j}(\xi)
    =
    \max_{1\le j\le m}\widehat{\theta}_x^{A_j}(\xi).
\]
Consequently,
\[
    \dimcross\{x\}_{\bigcup_{j=1}^m A_j}
    =
    \sup_{\Xi\in\mathcal I_x(\bigcup_j A_j)}
    \sum_{\xi\in\Xi}
    \max_{1\le j\le m}\widehat{\theta}_x^{A_j}(\xi).
\]
\end{theorem}

\begin{proof}
Monotonicity gives one support inclusion. For the other, a realizing sequence
in the finite union has a subsequence in one component. Now fix a probe
\(\gamma\) in a channel \(\xi\) of the union. Finite stability gives
\[
\begin{aligned}
    \Pbox\{x\}_{(\bigcup_j A_j)\cap\Gamma_\gamma}
    &=
    \Pbox\{x\}_{\bigcup_j(A_j\cap\Gamma_\gamma)}\\
    &=
    \max_{1\le j\le m}
    \Pbox\{x\}_{A_j\cap\Gamma_\gamma}.
\end{aligned}
\]
For each \(j\), either the corresponding trace intersection is not recurrent
at \(x\) and has Point-Box value zero, or the same probe is admissible for
\(A_j\) in the channel \(\xi\). Hence
\[
    \Pbox\{x\}_{A_j\cap\Gamma_\gamma}
    \le
    \widehat{\theta}_x^{A_j}(\xi).
\]
Taking the probe supremum proves one weight inequality. Monotonicity under
\(A_j\subset\bigcup_\ell A_\ell\) proves the reverse one. Channels outside
the union support give zero on both sides.

The variational formula is the definition of the Point-Cross Dimension.

\end{proof}

If \(\Xi_j\in\mathcal I_x(A_j)\) are pairwise disjoint and their union is
projectively independent, then
\[
    \dimcross\{x\}_{\bigcup_{j=1}^m A_j}
    \ge
    \sum_{j=1}^m\sum_{\xi\in\Xi_j}\theta_x^{A_j}(\xi).
\]
Indeed, the union is an admissible family and each of its weights dominates
the corresponding component weight. The disjointness condition prevents
counting a repeated channel twice. Thus repeated channels merge through a
maximum, while jointly independent component channels may contribute additively.

The following construction places sets with prescribed local Point-Box
dimensions on finitely many independent rays.

\begin{definition}[Axial fractal frame]
\label{def:axial-fractal-frame}
Let \(x\in\R^n\), let \(1\le k\le n\), and let
\(v_1,\ldots,v_k\in S^{n-1}\) be linearly independent. For each \(i\), let
\(F_i\subset[0,\rho_i]\), where \(\rho_i>0\), satisfy
\(0\in F_i\) and \(0\in\overline{F_i\setminus\{0\}}\).
The set
\[
    \mathcal F
    :=
    \bigcup_{i=1}^k\bigl(x+F_i v_i\bigr)
\]
is the \emph{axial fractal frame} generated by
\((F_1,v_1),\ldots,(F_k,v_k)\) at \(x\). Set
\(d_i:=\Pbox\{0\}_{F_i}\).
\end{definition}

\begin{proposition}[Weights of an axial fractal frame]
\label{prop:point-cross-fractal-frame}
For the frame \(\mathcal F\) of
Definition~\ref{def:axial-fractal-frame},
\[
    \mathbb P\Sigma_x(\mathcal F)
    =
    \{[v_1],\ldots,[v_k]\},
    \qquad
    \theta_x^{\mathcal F}([v_i])=d_i.
\]
Consequently,
\[
    \dimcross\{x\}_{\mathcal F}
    =
    \sum_{i=1}^k d_i,
    \qquad
    \Pbox\{x\}_{\mathcal F}
    =
    \max_{1\le i\le k}d_i,
    \qquad
    \Ptan\{x\}_{\mathcal F}=k.
\]
\end{proposition}

\begin{proof}
Set \(A_i:=x+F_i v_i\). Its projective tangent support is
\(\{[v_i]\}\). The straight probe \(t\mapsto x+tv_i\) has trace intersection
\(A_i\), so the isometry \(t\mapsto x+tv_i\) gives
\(\theta_x^{A_i}([v_i])\ge d_i\). Conversely, every trace intersection with
\(A_i\) is a subset of \(A_i\), and monotonicity gives the reverse inequality.
Thus \(\theta_x^{A_i}([v_i])=d_i\).

Theorem~\ref{thm:exact-finite-union-calculus} now identifies the support and
the weights of \(\mathcal F=\bigcup_i A_i\). Since the channels are
projectively independent, their weights add. The finite-union formula for
\(\Pbox\) gives the maximum of the arm exponents, and the span of the tangent
support has dimension \(k\).
\end{proof}

The same formula applies to curvilinear frames. Indeed, let
\(\Phi:U\to V\) be a \(C^1\)-diffeomorphism near \(x\), set
\(y:=\Phi(x)\), and put \(\mathcal C:=\Phi(\mathcal F\cap U)\). Channelwise
covariance, followed by local bi-Lipschitz invariance for Point-Box, gives
\[
\begin{gathered}
    \mathbb P\Sigma_y(\mathcal C)
    =\{[D\Phi(x)v_1],\ldots,[D\Phi(x)v_k]\},
    \qquad
    \theta_y^{\mathcal C}([D\Phi(x)v_i])=d_i,\\
    \dimcross\{y\}_{\mathcal C}=\sum_{i=1}^k d_i,
    \qquad
    \Pbox\{y\}_{\mathcal C}=\max_{1\le i\le k}d_i,
    \qquad
    \Ptan\{y\}_{\mathcal C}=k.
\end{gathered}
\]

\subsection{Intersections}
\label{subsec:directional-intersections}

Unlike finite unions, intersections have no exact profile formula without
additional geometric hypotheses.

\begin{proposition}[Common-channel upper bound]
\label{prop:point-cross-intersection-upper-bound}
Let \(A,B\subset\R^n\) and \(x\in\overline{A\cap B}\). Then
\[
    \mathbb P\Sigma_x(A\cap B)
    \subset
    \mathbb P\Sigma_x(A)\cap\mathbb P\Sigma_x(B),
\]
and, for every \(\xi\in\mathbb P\Sigma_x(A\cap B)\),
\[
    \theta_x^{A\cap B}(\xi)
    \le
    \min\{\theta_x^A(\xi),\theta_x^B(\xi)\}.
\]
In particular,
\[
    \dimcross\{x\}_{A\cap B}
    \le
    \min\{\dimcross\{x\}_A,\dimcross\{x\}_B\}.
\]
\end{proposition}

\begin{proof}
All assertions follow from the inclusions \(A\cap B\subset A,B\) and the
monotonicity of tangent supports, directional weights, and the Point-Cross
Dimension.
\end{proof}

\begin{proposition}[Submodularity]
\label{prop:point-cross-submodularity}
For all \(A,B\subset\R^n\) and \(x\in\R^n\),
\[
 \dimcross\{x\}_{A\cup B}+\dimcross\{x\}_{A\cap B}
 \le\dimcross\{x\}_A+\dimcross\{x\}_B.
\]
In particular, Point-Cross is subadditive under finite unions.
\end{proposition}

\begin{proof}
For each \(0<t\le1\), the finite-union profile identity gives
\(V_x^{A\cup B}(t)=V_x^A(t)+V_x^B(t)\), while monotonicity gives
\(V_x^{A\cap B}(t)\subseteq V_x^A(t)\cap V_x^B(t)\).
The union identity also holds when \(x\notin\overline{A\cup B}\), since
all these spaces are zero. Grassmann's identity yields
\[
 r_x^{A\cup B}(t)+r_x^{A\cap B}(t)
 \le r_x^A(t)+r_x^B(t).
\]
Integrate and apply Proposition~\ref{prop:threshold-integral}.
Nonnegativity and induction give finite subadditivity.
\end{proof}

The common-channel inequalities can be strict even for smooth germs: the line
\(A=\{(t,0):t\in\R\}\) and the parabola
\(B=\{(t,t^2):t\in\R\}\) have the same unit-weight channel at the origin,
but their intersection germ has Point-Cross Dimension zero.

Clean intersections recover Grassmann's identity. More precisely, let
\(M,N\subset\R^n\) be \(C^1\) embedded submanifolds without boundary,
containing \(x\), of dimensions \(p\) and \(q\). If \(M\cap N\) is a
\(C^1\) submanifold near \(x\) and
\(T_x(M\cap N)=T_xM\cap T_xN\), then smooth saturation and the finite-union
formula give
\[
\begin{gathered}
    \dimcross\{x\}_{M\cup N}=\dim(T_xM+T_xN),
    \qquad
    \dimcross\{x\}_{M\cap N}=\dim(T_xM\cap T_xN),\\
    \dimcross\{x\}_{M\cup N}
    +\dimcross\{x\}_{M\cap N}=p+q.
\end{gathered}
\]
In particular, transversality gives the values \(n\) and \(p+q-n\),
respectively.

\subsection{Exact Cartesian-product additivity}
\label{subsec:exact-cartesian-products}

Coordinate channels inherit the weights of their factors. A mixed channel
has weight no larger than either nonzero projection. These bounds identify
the threshold subspaces as direct sums and give scalar additivity.

\begin{proposition}[Projection bounds and coordinate inheritance]
\label{prop:product-channel-projection}
Let \(A\subset\R^m\), \(B\subset\R^n\), \(x\in\overline A\), and
\(y\in\overline B\). Let
\(\eta=[(v,w)]\in\mathbb P\Sigma_{(x,y)}(A\times B)\).
If \(v\ne0\), then \([v]\in\mathbb P\Sigma_x(A)\) and
\[
    \theta_{(x,y)}^{A\times B}(\eta)
    \le
    \theta_x^A([v]).
\]
If \(w\ne0\), then \([w]\in\mathbb P\Sigma_y(B)\) and
\[
    \theta_{(x,y)}^{A\times B}(\eta)
    \le
    \theta_y^B([w]).
\]
Moreover, the coordinate channels satisfy
\[
    \theta_{(x,y)}^{A\times B}([(v,0)])
    =
    \theta_x^A([v])
\]
for every \([v]\in\mathbb P\Sigma_x(A)\), and
\[
    \theta_{(x,y)}^{A\times B}([(0,w)])
    =
    \theta_y^B([w])
\]
for every \([w]\in\mathbb P\Sigma_y(B)\).
\end{proposition}

\begin{proof}
Fix \(\gamma\in\mathcal G_{(x,y)}^{A\times B}(\eta)\) and write
\(\gamma(t)=(\alpha(t),\beta(t))\). If \(v\ne0\), then
\(\alpha'(0)\) is a nonzero multiple of \(v\). Choose a linear functional
\(\ell\) on \(\R^m\) with \(\ell(\alpha'(0))>0\). After restricting the
parameter interval, which preserves admissibility and the local Point-Box
value of the trace intersection, there is \(c>0\) such that
\[
    (\ell\circ\alpha)'(t)\ge c.
\]
Thus \(\alpha\) is injective and regular, and
\[
    \|\alpha(t)-\alpha(s)\|
    \ge
    \frac{c}{\|\ell\|}|t-s|.
\]
Since \(\gamma\) is \(C^1\) on a compact interval, it is Lipschitz there:
\[
    \|\gamma(t)-\gamma(s)\|\le M|t-s|
\]
for some \(M>0\). Coordinate projection is \(1\)-Lipschitz. Conversely,
\[
    \|\gamma(t)-\gamma(s)\|
    \le
    \frac{M\|\ell\|}{c}
    \|\alpha(t)-\alpha(s)\|.
\]
Thus the restriction of \(\pi_A\) to \(\Gamma_\gamma\) is a bi-Lipschitz
bijection onto \(\Gamma_\alpha\).

Set \(E:=(A\times B)\cap\Gamma_\gamma\). Recurrence of \(E\) implies that
\(\alpha\) is recurrent in \(A\), so \([v]\in\mathbb P\Sigma_x(A)\).
Let \(p_0:=(x,y)\) and \(\Psi:=\pi_A|_{\Gamma_\gamma}\).
The restriction of \(\Psi\) is a bi-Lipschitz bijection from
\(E\cup\{p_0\}\) onto \(\pi_A(E)\cup\{x\}\), with \(\Psi(p_0)=x\).
Proposition~\ref{prop:pbox-basic-package} therefore applies even when
\(p_0\notin E\). Together with
\(\pi_A(E)\subset A\cap\Gamma_\alpha\), it gives
\[
\begin{aligned}
    \Pbox\{(x,y)\}_E
    &=
    \Pbox\{x\}_{\pi_A(E)}\\
    &\le
    \Pbox\{x\}_{A\cap\Gamma_\alpha}\\
    &\le
    \theta_x^A([v]).
\end{aligned}
\]
Taking the supremum over \(\gamma\) proves the first projection bound. The
second is symmetric.

The projection bounds give the upper estimates for coordinate channels. For
the reverse estimates, closure invariance and
\(\overline{A\times B}=\overline A\times\overline B\) allow us to replace
both factors and their product by their closures. If
\(\alpha\in\mathcal G_x^{\overline A}([v])\), then
\(\widetilde\alpha(t):=(\alpha(t),y)\)
is a regular probe for \(\overline A\times\overline B\) in the coordinate
channel \([(v,0)]\), and
\[
    (\overline A\times\overline B)\cap\Gamma_{\widetilde\alpha}
    =
    (\overline A\cap\Gamma_\alpha)\times\{y\}.
\]
The two trace intersections are isometric. Taking suprema and returning from
the closures gives
\[
\begin{aligned}
    \theta_{(x,y)}^{A\times B}([(v,0)])
    &=
    \theta_{(x,y)}^{\overline A\times\overline B}([(v,0)])\\
    &\ge
    \theta_x^{\overline A}([v])
    =
    \theta_x^A([v]).
\end{aligned}
\]
Together with the projection bound, this proves the first coordinate
equality. The second is symmetric.
\end{proof}

When both \(v\) and \(w\) are nonzero, the proposition gives the useful
combined estimate
\[
    \theta_{(x,y)}^{A\times B}([(v,w)])
    \le
    \min\{\theta_x^A([v]),\theta_y^B([w])\}.
\]
These bounds give an exact identity at every positive threshold.

\begin{proposition}[Threshold subspaces of a product]
\label{prop:product-threshold-spaces}
Under the hypotheses of Proposition~\ref{prop:product-channel-projection},
for every \(0<t\le1\),
\[
 V_{(x,y)}^{A\times B}(t)
 =\bigl(V_x^A(t)\times\{0\}\bigr)
  \oplus\bigl(\{0\}\times V_y^B(t)\bigr),
 \qquad
 r_{(x,y)}^{A\times B}(t)=r_x^A(t)+r_y^B(t).
\]
\end{proposition}

\begin{proof}
If \([(v,w)]\) has product weight at least \(t\), each nonzero coordinate
projection has factor weight at least \(t\). Thus \((v,w)\) belongs to
\(\bigl(V_x^A(t)\times\{0\}\bigr)
\oplus\bigl(\{0\}\times V_y^B(t)\bigr)\), which proves one inclusion after
taking spans. Conversely, coordinate inheritance puts the lifts of every factor
channel of weight at least \(t\) in the product superlevel set. These lifts
generate the direct sum. Taking dimensions proves the rank identity.
\end{proof}

\begin{theorem}[Exact Cartesian-product additivity]
\label{thm:exact-product}
Let \(A\subset\R^m\), \(B\subset\R^n\), \(x\in\overline A\), and
\(y\in\overline B\). Then
\[
    \boxed{
    \dimcross\{(x,y)\}_{A\times B}
    =
    \dimcross\{x\}_A+\dimcross\{y\}_B.
    }
\]
\end{theorem}

\begin{proof}
Integrating the rank identity in
Proposition~\ref{prop:product-threshold-spaces} and applying
Proposition~\ref{prop:threshold-integral} gives the result.
\end{proof}

By induction, if \(A_j\subset\R^{n_j}\) and
\(x_j\in\overline{A_j}\) for \(1\le j\le N\), then
\[
    \dimcross\{(x_1,\ldots,x_N)\}_{A_1\times\cdots\times A_N}
    =
    \sum_{j=1}^N\dimcross\{x_j\}_{A_j}.
\]
This equality does not require saturation of the factors. Mixed product
channels may still occur, but the theorem does not assert a channelwise
formula for their weights. The threshold subspace identity shows that such channels cannot increase
the scalar value beyond the sum of the factor values.

\begin{remark}[External products and coupled spatial realization]
\label{rem:external-products-internal-spatialization}
Theorem~\ref{thm:exact-product} is an additivity statement for the
Point-Cross Dimension of an external Cartesian product. It does not imply
that independent channels in one germ force comparable Point-Box complexity.
Axial fractal frames already show the opposite: their channel weights add,
whereas their Point-Box Dimension retains only the largest arm exponent.

The channel weights alone therefore do not guarantee an additive Point-Box
lower bound. One possible approach is to impose further hypotheses on
combinations of points in different directions at common scales. Finding
sufficient hypotheses remains an open problem.

Subsection~\ref{subsec:skeletons-versus-spatialized-products} illustrates this
distinction through cross/cell and Cantor-cross/Cantor-dust configurations.
\end{remark}

\section{Power-rate directional accessibility and sharpness}
\label{sec:quantitative-accessibility}

Tangent occurrence is qualitative: it requires approximation along a sequence
approaching the base point. A uniform power-rate approximation along a ray
forces the quantitative lower bound \(\theta_x^A(\xi)\ge1-1/d\).
We prove this by applying Lemma~\ref{lem:perturbative-sequential-reconstruction}
to power-law sequences whose longitudinal gaps dominate the approximation
error. The same sequences show that the bound is sharp for every \(d>1\).

\begin{definition}[Power-rate directional accessibility]
\label{def:power-rate-directional-accessibility}
Let \(A\subset\R^n\), let \(x\in\overline A\), let
\(\xi\in\mathbb{RP}^{n-1}\), and let \(d>1\). We say that \(\xi\) is
\emph{power-rate accessible of order \(d\) for \(A\) at \(x\)} if there
exist a unit representative \(v\in\xi\) and constants \(C,t_0>0\) such that
\[
    \operatorname{dist}(x+tv,A)
    \le
    Ct^d
    \qquad
    (0<t<t_0).
\]
\end{definition}

After translating \(x\) to zero, this hypothesis says that the ray germ
lies in the sea-tangle neighborhood
\(ST_d(A-x,C)=\{z:\operatorname{dist}(z,A-x)\le C\|z\|^d\}\)
of~\cite[Definition~4.1]{KoikePaunescu2009}. The result below estimates
the directional weight from that approximation rate.

Only one orientation of the projective direction is required. The condition
is stronger than tangent occurrence, which needs approximation only along a
sequence, and stronger than the uniform first-order relation
\(\operatorname{dist}(x+tv,A)=o(t)\).

The following metric calculation will be used both to obtain the lower bound
and to prove its sharpness. In this section, \(X\asymp Y\) denotes
two-sided comparison by positive constants independent of the limiting index
or covering scale, and \(X\lesssim Y\) denotes the corresponding upper
comparison.

\begin{lemma}[Local Point-Box dimension of power-law sequences]
\label{lem:power-sequence-pbox}
For every \(\beta>0\), the set
\(S_\beta:=\{0\}\cup\{j^{-\beta}:j\ge1\}\subset\R\) satisfies
\[
    \Pbox\{0\}_{S_\beta}=\dimBox(S_\beta)=\frac1{\beta+1}.
\]
The same local value holds after deleting \(0\) or finitely many positive
terms.
\end{lemma}

\begin{proof}
Put \(a_j:=j^{-\beta}\) and \(\Delta_j:=a_j-a_{j+1}\).
The mean value theorem gives \(\Delta_j\asymp j^{-\beta-1}\), and
convexity of \(t\mapsto t^{-\beta}\) shows that these gaps decrease.
For sufficiently small \(\varepsilon>0\), let \(J_\varepsilon\) be the
largest index with \(\Delta_{J_\varepsilon}\ge2\varepsilon\). Then
\[
    J_\varepsilon\asymp\varepsilon^{-1/(\beta+1)},
    \qquad
    a_{J_\varepsilon}\asymp\varepsilon^{\beta/(\beta+1)}.
\]
The first \(J_\varepsilon\) positive terms are pairwise at distance at
least \(2\varepsilon\), so a cover by sets of diameter at most
\(\varepsilon\) needs at least \(J_\varepsilon\) sets. Conversely, cover
these terms individually and cover the remaining interval
\([0,a_{J_\varepsilon}]\) by intervals of length \(\varepsilon\). This gives
\[
    J_\varepsilon
    \le N_\varepsilon(S_\beta)
    \le J_\varepsilon+\left\lceil\frac{a_{J_\varepsilon}}{\varepsilon}\right\rceil+1
    \asymp\varepsilon^{-1/(\beta+1)}.
\]
Every fixed neighborhood of \(0\) contains all but finitely many terms.
Deleting those terms subtracts only a fixed number from the lower covering
bound, while the upper bound is unchanged. Thus every such local intersection
has upper box dimension \(1/(\beta+1)\), and the assertion follows.
Deleting \(0\) changes any covering number by at most one.
\end{proof}

\begin{theorem}[Power-rate directional lower bound]
\label{thm:power-rate-directional-lower-bound}
Let \(A\subset\R^n\), let \(x\in\overline A\), and let
\(\xi\in\mathbb{RP}^{n-1}\). If \(\xi\) is power-rate accessible of order
\(d>1\), then
\[
    \xi\in\mathbb P\Sigma_x(A)
    \qquad\text{and}\qquad
    \boxed{
        \theta_x^A(\xi)
        \ge
        1-\frac1d.
    }
\]
\end{theorem}

\begin{proof}
Choose \(v\in\xi\) and \(C,t_0>0\) as in
Definition~\ref{def:power-rate-directional-accessibility}, and fix
\(\beta>1/(d-1)\). For all sufficiently large \(j\), set
\[
    t_j:=j^{-\beta},\qquad p_j:=x+t_jv,
\]
so that \(0<t_j<t_0\), and choose \(q_j\in A\) with
\(\|q_j-p_j\|<(C+1)t_j^d\). Write \(e_j:=q_j-p_j\) and
\(g_j:=t_j-t_{j+1}\). Since \(g_j\asymp j^{-\beta-1}\),
\[
    \frac{\|e_j\|+\|e_{j+1}\|}{g_j}
    =O\!\left(j^{1-\beta(d-1)}\right)\longrightarrow0.
\]
The transverse coordinates of the \(p_j\) vanish, so
Lemma~\ref{lem:perturbative-sequential-reconstruction} applies. For a
sufficiently large \(N\), it gives a regular probe \(\gamma\) in direction
\(v\) through \(Q_N:=\{q_j:j\ge N\}\), with
\[
    \Pbox\{x\}_{Q_N}=\Pbox\{x\}_{\{p_j:j\ge N\}}
    =\frac1{\beta+1}.
\]
The last equality follows from Lemma~\ref{lem:power-sequence-pbox} and the
isometry \(t\mapsto x+tv\). The reconstruction lemma also ensures that the
\(q_j\), \(j\ge N\), are distinct, different from \(x\), and converge to \(x\).
Hence \(\gamma\in\mathcal G_x^A(v)\), so \(\xi\in\mathbb P\Sigma_x(A)\)
and
\[
    \theta_x^A(\xi)
    \ge\Pbox\{x\}_{A\cap\Gamma_\gamma}
    \ge\Pbox\{x\}_{Q_N}
    =\frac1{\beta+1}.
\]
Letting \(\beta\downarrow1/(d-1)\) proves the asserted bound. The strict
inequality for \(\beta\) is what makes the perturbation negligible relative
to the longitudinal gaps. No limiting probe is required.
\end{proof}

\begin{proposition}[Sharpness on power-law sequences]
\label{prop:power-rate-sharpness}
For every \(n\ge1\) and every \(d>1\), there exists a compact set
\(E_d\subset\R^n\) such that
the channel \([e_1]\) is power-rate accessible for \(E_d\) at the origin of
order \(d\), but of no order \(d'>d\), and
\[
    \boxed{
        \theta_0^{E_d}([e_1])
        =
        1-\frac1d.
    }
\]
Moreover,
\[
    \Pvec\{0\}_{E_d}=0,
    \qquad
    \Pbox\{0\}_{E_d}
    =
    \dimcross\{0\}_{E_d}
    =
    1-\frac1d,
    \qquad
    \Ptan\{0\}_{E_d}=1.
\]
In particular,
\[
    \Pvec\{0\}_{E_d}
    <
    \dimcross\{0\}_{E_d}
    <
    \Ptan\{0\}_{E_d}.
\]
\end{proposition}

\begin{proof}
Set
\[
    \beta:=\frac1{d-1},
    \qquad
    a_j:=j^{-\beta},
    \qquad
    E_d:=\{0\}\cup\{a_j e_1:j\ge1\}.
\]
Writing \(\Delta_j:=a_j-a_{j+1}\), the mean value theorem gives
\[
    \Delta_j\asymp j^{-\beta-1}.
\]
Since \(\beta d=\beta+1\) and \(a_{j+1}\asymp a_j\),
\[
    \Delta_j\asymp a_j^d\asymp t^d
    \qquad
    (a_{j+1}\le t\le a_j).
\]
Therefore
\[
    \operatorname{dist}(te_1,E_d)
    \le
    \Delta_j
    \lesssim
    t^d.
\]
Thus \([e_1]\) is power-rate accessible of order \(d\).

For the failure of every larger order, set
\(m_j:=(a_j+a_{j+1})/2\).
The nearest points of \(E_d\) to \(m_je_1\) are the two consecutive points
\(a_je_1\) and \(a_{j+1}e_1\). Hence
\[
    \operatorname{dist}(m_je_1,E_d)
    =
    \frac{\Delta_j}{2}
    \asymp
    m_j^d.
\]
If \(d'>d\), then
\[
    \frac{\operatorname{dist}(m_je_1,E_d)}{m_j^{d'}}
    \asymp
    m_j^{d-d'}
    \longrightarrow\infty.
\]
No estimate of order \(d'>d\) holds in the positive orientation. In the
opposite orientation, \(\operatorname{dist}(-te_1,E_d)=t\) for \(t>0\),
so no power-rate estimate of order greater than one holds there either.
Thus the projective channel \([e_1]\) is not power-rate accessible of any
order \(d'>d\).

By Lemma~\ref{lem:power-sequence-pbox}, applied through the isometry
\(t\mapsto te_1\),
\[
    \Pbox\{0\}_{E_d}
    =
    \frac1{\beta+1}
    =
    1-\frac1d.
\]

The only projective tangent channel of \(E_d\) at the origin is \([e_1]\).
The straight probe \(t\mapsto te_1\) contains \(E_d\), so
\(\theta_0^{E_d}([e_1])\ge\Pbox\{0\}_{E_d}\).
Every trace intersection with \(E_d\) is a subset of \(E_d\), and Point-Box
monotonicity gives the reverse inequality. Thus the weight has the asserted
value, and the one-channel support gives the formula for \(\dimcross\).
Finally, \(E_d\) contains no nondegenerate segment germ, while its tangent
cone spans \(\R e_1\). Hence
\[
    \Pvec\{0\}_{E_d}=0,
    \qquad
    \Ptan\{0\}_{E_d}=1.
\]
\end{proof}

In particular, every value in \((0,1)\) occurs as the weight of a single
tangent channel: given \(c\in(0,1)\), choose \(d=1/(1-c)\) in the preceding
construction.

\begin{corollary}[Independent power-rate accessible directions]
\label{cor:independent-power-rate-accessible-channels}
Let \(A\subset\R^n\), let \(x\in\overline A\), and let
\(\xi_1,\ldots,\xi_k\in\mathbb{RP}^{n-1}\) be projectively independent.
Assume that \(\xi_i\) is power-rate accessible for \(A\) at \(x\) of order
\(d_i>1\) for every \(1\le i\le k\). Then
\[
    \dimcross\{x\}_A
    \ge
    \sum_{i=1}^k
    \left(
        1-\frac1{d_i}
    \right).
\]
\end{corollary}

\begin{proof}
Theorem~\ref{thm:power-rate-directional-lower-bound} gives
\[
    \xi_i\in\mathbb P\Sigma_x(A),
    \qquad
    \theta_x^A(\xi_i)
    \ge
    1-\frac1{d_i}
    \qquad
    (1\le i\le k).
\]
The family is admissible in the defining supremum for the Point-Cross
Dimension, so summing these inequalities proves the claim.
\end{proof}

As an immediate limiting case, if a projective direction is power-rate
accessible of every order \(d>1\), with constants allowed to depend on \(d\),
then it is a tangent channel of weight one. No converse is asserted:
power-rate accessibility is a verifiable sufficient condition, not a
characterization of the weighted profile.

\begin{remark}[Dependence of the accessibility hypothesis on coordinates]
The weight is covariant under local \(C^1\)-diffeomorphisms, but an
order estimate along a straight ray is not transported as a straight-ray
estimate. Such a map generally sends the ray to a curved germ. The
accessibility condition is therefore a sufficient condition in the chosen
coordinates, not a claimed invariant characterization of the weight.
\end{remark}

\section{Discriminating examples}
\label{sec:discriminating-examples}

We compute the four pointwise quantities for unions of axial sets, Cartesian
products, oscillatory graphs, a Bedford\textendash{}McMullen carpet, and several
self-similar sets and cusps. These examples show both strict inequalities
and equalities between the quantities. They also distinguish channel weights
from local upper box dimension.

\subsection{Axial unions and Cartesian products}
\label{subsec:skeletons-versus-spatialized-products}

Independent weights do not prescribe how points in different directions
combine. Two elementary pairs make this distinction explicit.

\begin{example}[Crosses and products]
\label{cor:cross-versus-filled-cell}
\label{ex:cantor-cross-versus-dust}
For \(2\le k\le n\), let \(\mathcal C_k\) be the union of the \(k\)
coordinate segments \([0,1]e_i\), and let
\(\mathcal Q_k=[0,1]^k\times\{0\}^{n-k}\). Let \(K\) be the middle-third
Cantor set, put \(d=\log2/\log3\), and define
\(\mathcal X_K=(K\times\{0\})\cup(\{0\}\times K)\) and
\(\mathcal D_K=K\times K\). At the origin,
\[
\begin{array}{c@{\qquad}cccc}
 &\Pvec&\Pbox&\dimcross&\Ptan\\
 \mathcal C_k&k&1&k&k\\
 \mathcal Q_k&k&k&k&k\\
 \mathcal X_K&0&d&2d&2\\
 \mathcal D_K&0&2d&2d&2
\end{array}
\]
\end{example}

\begin{proof}
The axial-frame formula gives the union values, and product additivity
gives the Point-Cross values of the products. The cube has local upper box
dimension \(k\). For the Cantor dust,
\(N_{3^{-m}}(K\times K)\asymp4^m\), and every neighborhood of the origin
contains a scaled corner copy, giving local upper box dimension \(2d\).
A nondegenerate segment in the dust would project onto an interval in
\(K\), so its segment rank is zero. The coordinate channels give all
the stated tangent ranks.
\end{proof}

Each pair has the same three scalar directional values, but different local
upper box dimensions. Their full tangent supports need not agree. Only the
products contain every mixed choice of the coordinate parameters.

\subsection{Oscillatory graphs and intersections with regular curves}
\label{subsec:oscillatory-germs}

In the axial examples, straight probes attain the channel weights. For
oscillatory graphs, counting intersections with a regular \(C^1\) graph gives
upper bounds for channel weights. To prove weight one in certain channels, we
first construct a vanishing-slope trace containing nondegenerate arcs in every
neighborhood of the base point, then apply the equivalence with regular
\(C^1\) probes.

Fix \(t_0>0\). For \(\alpha\ge0\), let
\[
    A_\alpha
    :=
    \overline{
        \{(t,t^\alpha\sin(1/t)):0<t\le t_0\}
    }.
\]
For \(\alpha>0\), the closure adds only the origin. For \(\alpha=0\), it
also contains the vertical segment \(\{0\}\times[-1,1]\).

\begin{proposition}[Local upper box dimension and tangent support of the chirps]
\label{prop:chirp-dispersive-tangential-profile}
At the origin,
\[
    \Pbox\{0\}_{A_\alpha}
    =
    \begin{cases}
        \dfrac{3-\alpha}{2},&0\le\alpha<1,\\[0.6em]
        1,&\alpha\ge1,
    \end{cases}
\]
and
\[
    \mathbb P\Sigma_0(A_\alpha)
    =
    \begin{cases}
        \{[e_1]\},&\alpha>1,\\
        \{[(1,m)]:m\in[-1,1]\},&\alpha=1,\\
        \mathbb{RP}^{1},&0\le\alpha<1.
    \end{cases}
\]
Consequently, \(\Ptan\{0\}_{A_\alpha}=1\) for \(\alpha>1\), and
\(\Ptan\{0\}_{A_\alpha}=2\) for \(0\le\alpha\le1\).
Moreover,
\[
    \Pvec\{0\}_{A_\alpha}
    =
    \begin{cases}
        0,&\alpha>0,\\
        1,&\alpha=0.
    \end{cases}
\]
\end{proposition}

\begin{proof}
We use the standard chirp covering calculation~\cite{tricot}.
On the oscillation interval with index \(j\), the horizontal
width is comparable to \(j^{-2}\) and the vertical amplitude to
\(j^{-\alpha}\). At scale \(\varepsilon\), split at
\(J\asymp\varepsilon^{-1/2}\). For \(0\le\alpha<1\), summing the
vertical variations of the outer oscillations gives the upper bound
\[
    \frac1\varepsilon\sum_{j\le J}j^{-\alpha}
    \asymp
    \varepsilon^{-(3-\alpha)/2}
\]
for their covering number. Taking every other oscillation with
\(j\le cJ\), where \(c>0\) is sufficiently small, makes the selected
pieces \(\varepsilon\)-separated in the horizontal direction. Their vertical
projections therefore give
\(N_\varepsilon(A_\alpha)\gtrsim
\varepsilon^{-1}\sum_{j\le cJ}j^{-\alpha}
\asymp\varepsilon^{-(3-\alpha)/2}\).
The inner tail is contained in a rectangle of width \(O(J^{-1})\) and
height \(O(J^{-\alpha})\). A coordinate grid covers it by
\(O((1+J^{-1}/\varepsilon)(1+J^{-\alpha}/\varepsilon))
=O(\varepsilon^{-(3-\alpha)/2})\) sets of diameter \(O(\varepsilon)\).
Changing the scale by a constant does not affect the exponent. When
\(\alpha=1\), the outer sum is
\(O(\varepsilon^{-1}\log(1/\varepsilon))\) and the inner rectangle needs
\(O(\varepsilon^{-1})\) covering sets. For \(\alpha>1\), the outer sum is
\(O(\varepsilon^{-1})\), and the graph is rectifiable near the origin.
Every neighborhood contains a nondegenerate graph arc, so the local upper
box dimension is at least one in the last two cases. These estimates hold on
every tail when \(\alpha>0\).

For \(\alpha=0\), fix a sufficiently small \(r>0\). Near each zero
\(u=k\pi\), take the graph portion corresponding to the connected component
of \(\{u:|\sin u|\le r/2\}\) that contains \(k\pi\). For all sufficiently
large \(k\), this portion lies in \(B(0,r)\) and its second-coordinate
projection has length \(r\), so it requires at least a constant multiple of
\(r/\varepsilon\) covering sets of diameter \(\varepsilon\). If
\[
    k\le c_r\varepsilon^{-1/2}
\]
with \(c_r>0\) sufficiently small, the corresponding portions are separated
from one another by more than \(\varepsilon\) in the first coordinate. There
are \(\asymp\varepsilon^{-1/2}\) such portions, and therefore
\[
    N_\varepsilon\bigl(A_0\cap B(0,r)\bigr)
    \gtrsim
    \varepsilon^{-3/2}.
\]
The global covering estimate gives the matching local upper bound. Thus every
neighborhood has upper box dimension \(3/2\). The vertical segment added by
closure has dimension one and does not change this value. This gives the
stated Point-Box values.

The slope of a base-point secant is
\(t^{\alpha-1}\sin(1/t)\). It tends to zero when \(\alpha>1\), has the
interval \([-1,1]\) as its limit set when \(\alpha=1\), and realizes every
finite value when \(0\le\alpha<1\). In the last range the vertical direction
is obtained from successive contacts with the upper envelope \(y=t^\alpha\)
if \(\alpha>0\), and from the vertical
segment if \(\alpha=0\). This proves the support formula and the tangential
ranks. Any nonvertical line-segment germ through the origin would have the
form \(y=mt\) and would force \(t^{\alpha-1}\sin(1/t)=m\) throughout an
interval. This is impossible because the left-hand side vanishes at
\(t=(k\pi)^{-1}\) and is nonzero at intervening points. No vertical segment
occurs when \(\alpha>0\), whereas for \(\alpha=0\) the added vertical segment
gives exactly one independent exact direction. This proves the Point-Vector
formula.
\end{proof}

To obtain full weight, we use the following gluing lemma. It controls the
slopes between the selected points and the tangent directions along their arcs.

\begin{lemma}[Gluing arcs with converging tangent directions]
\label{lem:coherent-arc-packet-gluing}
Let \(A\subset\R^n\), let \(x\in\overline A\), and let \(v\in S^{n-1}\).
Suppose that \(A\) contains nondegenerate \(C^1\) embedded arcs \(C_j\)
through points
\[
    p_j=x+s_jv+z_j,
    \qquad
    z_j\in v^\perp,
\]
where \(s_j\downarrow0\), \(z_j\to0\), and
\[
    \frac{\|z_j-z_{j+1}\|}{s_j-s_{j+1}}
    \longrightarrow0.
\]
If the tangent lines \(T_{p_j}C_j\) converge to \(\R v\), then
\(\theta_x^A([v])=1\).
\end{lemma}

\begin{proof}
For all sufficiently large \(j\), a neighborhood of \(p_j\) in \(C_j\)
is a graph \(x+tv+f_j(t)\) over an interval containing \(s_j\), with
\(f_j(s_j)=z_j\) and
\(\operatorname{Lip}(f_j)\le\delta_j\), where \(\delta_j\to0\). Write
\(g_j:=s_j-s_{j+1}\). After further restriction, choose nondegenerate
intervals \(I_j=[a_j,b_j]\) such that
\[
    |a_j-s_j|+|b_j-s_j|
    \le
    \frac1j\min\{g_{j-1},g_j\}
    \qquad(j\ge2).
\]
For all large \(j\), these intervals are ordered and disjoint, since
\[
    a_j-b_{j+1}
    \ge
    g_j\left(1-\frac1j-\frac1{j+1}\right)>0.
\]
Discard finitely many initial arcs and relabel. Define \(f=f_j\) on each
\(I_j\), define it affinely on every gap \([b_{j+1},a_j]\), and set
\(f(0)=0\). Since \(z_j\to0\) and the endpoint perturbations tend to zero,
this function is continuous at the origin.

The slope of the affine connector across the \(j\)-th gap is at most
\[
\frac{
    \|z_j-z_{j+1}\|
    +\delta_j|a_j-s_j|
    +\delta_{j+1}|b_{j+1}-s_{j+1}|
}{
    g_j\left(1-\frac1j-\frac1{j+1}\right)
},
\]
which tends to zero. Given \(\eta>0\), all graph pieces and affine connectors
sufficiently near the origin have Lipschitz constant at most \(\eta\).
Any two positive points in a sufficiently short initial interval lie in a
finite union of adjacent pieces. Partitioning the interval at the intervening
junctions and summing the corresponding Lipschitz estimates gives
\[
    \|f(t)-f(s)\|\le\eta|t-s|.
\]
The case in which one endpoint is zero follows by approximation through the
junction points converging to zero. Hence \(f\) is \(\eta\)-Lipschitz on a
sufficiently short initial interval, and therefore
\[
    \operatorname{Lip}(f|_{[0,r]})\longrightarrow0
    \qquad(r\downarrow0).
\]
The graph
\(\Gamma_f=\{x+tv+f(t):0\le t\le b_1\}\) is therefore a
vanishing-slope trace.
Every neighborhood of \(x\) contains a selected nondegenerate arc contained
in \(A\cap\Gamma_f\). Hence every such neighborhood intersection has upper
box dimension at least one. Since \(A\cap\Gamma_f\) is contained in a
Lipschitz trace, its local upper box dimension is at most one. Therefore
\(\Pbox\{x\}_{A\cap\Gamma_f}=1\). The equivalence of regular \(C^1\) and
vanishing-slope weights from
Corollary~\ref{cor:c1-vs-weight-equivalence} gives the conclusion.
\end{proof}

\begin{theorem}[Exact weighted profile of the chirps]
\label{thm:chirp-profile}
The weighted profiles are as follows.
\begin{enumerate}
    \item If \(\alpha>1\), then
    \(\theta_0^{A_\alpha}([e_1])=1\).
    \item If \(\alpha=1\), then
    \[
        \theta_0^{A_1}([(1,1)])
        =
        \theta_0^{A_1}([(1,-1)])
        =1,
    \]
    while \(\theta_0^{A_1}([(1,m)])=1/2\) for \(-1<m<1\).
    \item If \(0\le\alpha<1\), then
    \(\theta_0^{A_\alpha}([e_2])=1\), while
    \(\theta_0^{A_\alpha}([(1,m)])=1/2\) for every \(m\in\R\).
\end{enumerate}
Consequently,
\[
    \dimcross\{0\}_{A_\alpha}
    =
    \begin{cases}
        1,&\alpha>1,\\[0.3em]
        2,&\alpha=1,\\[0.3em]
        \dfrac32,&0\le\alpha<1.
    \end{cases}
\]
In particular, for \(0<\alpha<1\),
\[
    \Pbox\{0\}_{A_\alpha}
    =
    \frac{3-\alpha}{2}
    <
    \dimcross\{0\}_{A_\alpha}
    =
    \frac32
    <
    \Ptan\{0\}_{A_\alpha}
    =2.
\]
\end{theorem}

\begin{proof}
We first treat the finite-slope channels. Let \(\gamma\) be a regular probe
in the channel \([(1,m)]\). Its first velocity component is nonzero, so after
restriction its trace is a graph \(y=g(t)\), where \(g(0)=0\) and
\(g'(0)=m\). Contacts with the chirp satisfy
\[
    \sin u=q(u),
    \qquad
    u=1/t,
    \qquad
    q(u):=u^\alpha g(1/u).
\]
Since \(g(t)=mt+o(t)\) and \(g'(t)\to m\), if
\(0\le\alpha<1\), then \(q(u)\to0\) and
\[
    q'(u)
    =
    \alpha u^{\alpha-1}g(1/u)
    -u^{\alpha-2}g'(1/u)
    \longrightarrow0.
\]
If \(\alpha=1\) and \(|m|<1\), then \(q(u)\to m\) and again
\(q'(u)\to0\). In either case, for all sufficiently large periods the roots
of \(\sin u-q(u)\) lie in finitely many fixed neighborhoods on which
\(|\cos u|\) is bounded below. Since \(q'(u)\to0\), the function is strictly
monotone on each such neighborhood. Thus each period contains at most a fixed
number of contacts.

Restrict the probe further so that all its contacts with positive first
coordinate correspond to these sufficiently large periods.
Fix a small \(\varepsilon>0\) and put
\(N:=\lceil\varepsilon^{-1/2}\rceil\). The periods with index at most \(N\)
contribute \(O(N)\) points. All contacts in later periods lie on a graph
segment of length \(O(1/N)\), which is covered by
\(O((N\varepsilon)^{-1})=O(\varepsilon^{-1/2})\) sets of diameter
\(\varepsilon\). Hence every finite-slope trace in the stated ranges has
local upper box dimension at most \(1/2\). The straight line \(y=mt\)
supplies the reverse bound. For \(0\le\alpha<1\), the intermediate value
theorem gives one root in a fixed neighborhood of each sufficiently large
multiple of \(\pi\). For \(\alpha=1\), the equation is \(\sin u=m\) and
gives two interlaced sequences. In either case, one obtains contact parameters
\(u_k\asymp k\) with consecutive gaps bounded above and below. Their
reciprocals \(t_k:=u_k^{-1}\) therefore satisfy
\[
    t_k\asymp k^{-1},
    \qquad
    t_k-t_{k+1}\asymp k^{-2}.
\]
To make the lower covering estimate explicit, fix a neighborhood of the
origin and choose \(k_0\) so that all contacts \((t_k,mt_k)\), \(k\ge k_0\),
lie in it and \(t_k-t_{k+1}\ge a k^{-2}\) for some \(a>0\).
Choose \(c>0\) with \(a/c^2>1\) and set
\(K_\varepsilon:=\lfloor c\varepsilon^{-1/2}\rfloor\). For sufficiently
small \(\varepsilon\), the contacts with
\(k_0\le k\le K_\varepsilon\) are pairwise more than \(\varepsilon\)
apart, since their first coordinates already have this separation.
Any cover by sets of diameter at most \(\varepsilon\) therefore requires
at least \(K_\varepsilon-k_0+1\asymp\varepsilon^{-1/2}\) sets.
Together with the preceding upper estimate, this gives local upper box
dimension \(1/2\) on the straight trace. Thus the finite-slope weights
equal \(1/2\). When
\(\alpha=0\), the vertical segment added by closure meets such a graph only at
the origin and does not affect the calculation.

It remains to obtain the full distinguished channels. Let
\(t_j=(\pi/2+2\pi j)^{-1}\) be the parameters of the contacts with the
upper envelope \(y=t^\alpha\), after discarding finitely many terms so that
\(t_j\le t_0\). These are maxima of the oscillatory factor, not stationary
points of the full graph: for \(h(t)=t^\alpha\sin(1/t)\),
\(h'(t_j)=\alpha t_j^{\alpha-1}\). For
\(0<\alpha<1\), take \(v=e_2\), longitudinal coordinates
\(s_j=t_j^\alpha\), and transverse coordinates \(z_j=t_je_1\). Then
\[
    \frac{\|z_j-z_{j+1}\|}{s_j-s_{j+1}}
    \asymp
    j^{\alpha-1}
    \longrightarrow0,
\]
and the tangent lines at these contacts converge to \(\R e_2\). The gluing
lemma gives vertical weight one. For \(\alpha=0\), the vertical segment gives
the same conclusion directly. When \(\alpha=1\), the positive and negative
envelope contacts lie exactly on the rays in the channels \([(1,1)]\) and
\([(1,-1)]\), so their transverse coordinates vanish. Their tangent lines
have the same respective directions, and the gluing lemma gives both endpoint
channels weight one. Finally, if \(\alpha>1\), use \(v=e_1\) at the upper
envelope contacts.
Here
\[
    \frac{|t_j^\alpha-t_{j+1}^\alpha|}{t_j-t_{j+1}}
    \asymp
    t_j^{\alpha-1}
    \longrightarrow0,
\]
and the tangent lines converge to \(\R e_1\). Another application of the
gluing lemma gives the remaining full channel.

These weights determine the Point-Cross Dimension. For \(\alpha>1\), only the
horizontal channel occurs. For \(\alpha=1\), the two independent endpoint
channels have total weight two. If \(0\le\alpha<1\), the vertical channel
and any finite-slope channel have total weight \(3/2\), whereas two
finite-slope channels have total weight one. No independent family in
\(\R^2\) has more than two members. This proves the displayed values, and
Proposition~\ref{prop:chirp-dispersive-tangential-profile} gives the strict
chain.
\end{proof}

The critical value \(\alpha=1\) gives a transition invisible to both
Point-Vector and Point-Box Dimension:
\[
\bigl(
    \Pvec\{0\}_{A_\alpha},
    \Pbox\{0\}_{A_\alpha},
    \dimcross\{0\}_{A_\alpha},
    \Ptan\{0\}_{A_\alpha}
\bigr)
=
\begin{cases}
    (0,1,1,1),&\alpha>1,\\
    (0,1,2,2),&\alpha=1.
\end{cases}
\]
Thus the Point-Vector and Point-Box values remain unchanged, while the tangent
support acquires a second independent direction and the Point-Cross Dimension
doubles.

The chirps are not axial frames: the fractional value comes from a uniform
restriction on how often a regular \(C^1\) graph can meet successive
oscillations. In the full-weight channels, the selected arcs satisfy the
position and tangent estimates of Lemma~\ref{lem:coherent-arc-packet-gluing}.

\subsection{An anisotropic Bedford\textendash{}McMullen carpet}
\label{subsec:bedford-mcmullen-carpet}

We next consider a standard self-affine class introduced independently by
Bedford and McMullen~\cite{Bedford1984,McMullen1984}. Define
\[
    S_{i,j}(u,w)
    :=
    \left(\frac{u+i}{3},\frac{w+j}{2}\right),
    \qquad
    (i,j)\in\mathcal D:=\{(0,0),(1,0),(0,1)\},
\]
and let \(\Lambda\subset[0,1]^2\) be the attractor of these three maps. Put
\(d:=\log 2/\log 3\). The origin is the fixed point of \(S_{0,0}\).

The carpet has a full vertical edge but only a Cantor horizontal projection.
The interaction between these two features and the anisotropic contractions
determines all four pointwise invariants at the origin.

\begin{theorem}[Exact Point-Cross value for the carpet]
\label{thm:bedford-mcmullen-profile}
The projective tangent support of \(\Lambda\) at the origin is the closed arc
of first-quadrant directions, and
\[
    \theta_0^\Lambda([e_2])=1,
    \qquad
    \theta_0^\Lambda([e_1])=d.
\]
Every nonvertical channel \(\xi\in\mathbb P\Sigma_0(\Lambda)\) satisfies
\(\theta_0^\Lambda(\xi)\le d\). Consequently,
\[
    \Pvec\{0\}_\Lambda=1,
    \qquad
    \Pbox\{0\}_\Lambda=2-d,
    \qquad
    \dimcross\{0\}_\Lambda=1+d,
    \qquad
    \Ptan\{0\}_\Lambda=2.
\]
In particular,
\[
    \Pvec\{0\}_\Lambda
    <
    \Pbox\{0\}_\Lambda
    <
    \dimcross\{0\}_\Lambda
    <
    \Ptan\{0\}_\Lambda.
\]
\end{theorem}

\begin{proof}
Every point of \(\Lambda\) has an address
\(((i_1,j_1),(i_2,j_2),\ldots)\in\mathcal D^{\mathbb N}\), with coordinates
\[
    u=\sum_{\ell\ge1}i_\ell 3^{-\ell},
    \qquad
    w=\sum_{\ell\ge1}j_\ell 2^{-\ell}.
\]
Let \(C\) be the self-similar set with ternary digits in \(\{0,1\}\).
The address description gives
\[
    \pi_1(\Lambda)=C,
    \qquad
    \Lambda\cap(\R\times\{0\})=C\times\{0\},
    \qquad
    \Lambda\cap(\{0\}\times\R)=\{0\}\times[0,1].
\]
Thus the vertical channel has weight one and is an exact direction. The
horizontal straight probe has trace intersection \(C\times\{0\}\), whose
local Point-Box dimension at the origin is \(d\), so the horizontal weight is
at least \(d\). Here locality follows because every neighborhood of the
origin contains a scaled leftmost copy of \(C\), while the full set gives the
matching upper bound.

No exact direction with nonzero first component can occur, since its segment
germ would project onto an interval contained in \(C\). Hence
\(\Pvec\{0\}_\Lambda=1\). The tangent support, however, contains every
first-quadrant direction. Indeed, for a finite slope \(s\ge0\), set
\(y_k:=s(2/3)^{k+1}\) for all sufficiently large \(k\), and define
\[
    p_k
    :=
    S_{0,0}^{k}S_{1,0}(0,y_k)
    =
    \left(3^{-(k+1)},2^{-(k+1)}y_k\right).
\]
Since \((0,y_k)\in\Lambda\), one has \(p_k\in\Lambda\),
\(p_k\to0\), and the slope of \(p_k\) is exactly \(s\). The vertical
direction comes from the left edge, while containment in the first quadrant
gives the reverse inclusion. Therefore \(\Ptan\{0\}_\Lambda=2\).

Let \(\xi\) be nonvertical and let \(\gamma\) be a regular probe in this
channel. After restriction, its trace is a \(C^1\) graph over the first
coordinate, and horizontal projection is bi-Lipschitz on that trace. By the
pointed bi-Lipschitz invariance in
Proposition~\ref{prop:pbox-basic-package}, if
\(E:=\Lambda\cap\Gamma_\gamma\), then
\[
    \Pbox\{0\}_E
    =
    \Pbox\{0\}_{\pi_1(E)}.
\]
Since \(\pi_1(E)\subset C\),
\[
    \Pbox\{0\}_{\pi_1(E)}
    \le
    \dimBox(C)
    =d.
\]
This proves the uniform channel bound and, together with the horizontal lower
bound, gives \(\theta_0^\Lambda([e_1])=d\).

It remains to compute the local Point-Box value of the carpet. At vertical
scale \(r:=2^{-k}\), set \(\ell_k:=\lfloor kd\rfloor\), so that
\(r\le 3^{-\ell_k}<3r\) and \(\ell_k/k\to d\). An approximate square
is determined by the first \(\ell_k\) full symbols and by the row symbols
from level \(\ell_k+1\) to level \(k\). There are
\[
    3^{\ell_k}2^{k-\ell_k}
\]
such sets, all nonempty: any prescribed row symbol beyond level \(\ell_k\)
can be paired with column symbol zero, and the address can then be continued
with \((0,0)\). The prescribed column and row prefixes determine distinct
rectangles in the grid of width \(3^{-\ell_k}\) and height \(r\).
Their interiors are disjoint, although their boundaries may meet.
A set of diameter at most \(r\) projects onto coordinate sets of diameter
at most \(r\), and hence meets at most three grid intervals in each
coordinate. It therefore meets at most nine of these rectangles, including
boundary contacts. Conversely, each rectangle has width less than \(3r\)
and height \(r\), so it is covered by a uniformly bounded number of sets
of diameter at most \(r\). Since each approximate square is nonempty and
contained in its associated rectangle, these two bounds give
\[
    N_{2^{-k}}(\Lambda)
    \asymp
    3^{\ell_k}2^{k-\ell_k}.
\]
Indeed, the logarithmic quotient equals
\((\ell_k/k)(\log3/\log2)+1-\ell_k/k+O(1/k)\), which tends to \(2-d\).
Monotonicity between consecutive dyadic scales therefore gives
\(\dimBox(\Lambda)=2-d\).
For every neighborhood of the origin, a sufficiently small corner copy
\(S_{0,0}^{j}(\Lambda)\) lies inside that neighborhood. This copy is a
linear bi-Lipschitz image of \(\Lambda\), so the global exponent occurs in
every neighborhood and \(\Pbox\{0\}_\Lambda=2-d\).

Finally, an independent family contains at most two channels. If it contains
the vertical channel, its weight is at most \(1+d\). Otherwise its weight is
at most \(2d<1+d\). The independent vertical and horizontal channels attain
\(1+d\). The strict chain follows from \(1/2<d<1\).
\end{proof}

The tangent support of \(\Lambda\) is a continuum, but the vertical channel
has weight one and every other channel has weight at most \(d<1\).
The exact Point-Cross Dimension is
therefore not the tangential rank multiplied by a common channel weight.
Unlike the axial examples, the carpet also contains mixed positions at every
generation, and its Point-Box value is governed by the two contraction rates.
No equality is needed for the remaining finite-slope channels: their uniform
upper bound already determines the optimizing independent pair.

\subsection{Sierpiński sets and a filled cusp}
\label{subsec:sierpinski-and-cusp}

At basic-cell vertices, the Sierpiński sets have two independent channels
of weight one and noninteger local upper box dimension. A filled cusp has
local upper box dimension two but only one tangent channel, of weight one.

Let \(G\) be the standard Sierpiński gasket and \(S\) the standard
Sierpiński carpet. Denote by \(V_G\) and \(V_S\) the
vertices of their basic cells over all finite generations.

\begin{proposition}[Dense saturation in the Sierpiński models]
\label{prop:sierpinski-dense-directional-saturation}
The sets \(V_G\) and \(V_S\) are dense in \(G\) and
\(S\), respectively. For \(x\in V_G\),
\[
    \Pbox\{x\}_G=\frac{\log 3}{\log 2},
    \qquad
    \Pvec\{x\}_G=\dimcross\{x\}_G=\Ptan\{x\}_G=2,
\]
whereas for \(x\in V_S\),
\[
    \Pbox\{x\}_S=\frac{\log 8}{\log 3},
    \qquad
    \Pvec\{x\}_S=\dimcross\{x\}_S=\Ptan\{x\}_S=2.
\]
\end{proposition}

\begin{proof}
Basic-cell vertices are dense because every point lies in basic cells of
arbitrarily small diameter. Fix such a vertex \(x\) and one basic cell having
\(x\) as a corner. Repeatedly choosing the retained corner subcell at \(x\)
produces similar copies of the full attractor, with diameters tending to zero.
Every neighborhood of \(x\) therefore contains a similar copy of the whole
set. Similarity invariance gives the global box dimension as a lower bound for
each local intersection, while monotonicity gives the reverse bound. Using the
classical dimensions~\cite{Falconer2003} proves the two Point-Box formulas.
The two sides of the chosen basic cell incident to \(x\) remain in the
attractor. They therefore provide two independent exact directions, each of
weight one. The ambient planar bound completes the directional identities.
\end{proof}

For \(q>1\), define the filled power cusp
\[
    \mathcal K_q
    :=
    \{(t,y)\in\R^2:0\le t\le1,\ |y|\le t^q\}.
\]

\begin{proposition}[A filled cusp with one tangent channel]
\label{prop:dispersion-dominant-polynomial-cusp}
At the origin,
\[
    \Pbox\{0\}_{\mathcal K_q}=2,
    \qquad
    \Pvec\{0\}_{\mathcal K_q}
    =
    \dimcross\{0\}_{\mathcal K_q}
    =
    \Ptan\{0\}_{\mathcal K_q}
    =1.
\]
\end{proposition}

\begin{proof}
Every neighborhood intersection contains a nonempty planar open set, so its
upper box dimension is two. If \((t_j,y_j)\in\mathcal K_q\setminus\{0\}\)
tends to the origin, then \(|y_j/t_j|\le t_j^{q-1}\to0\). Thus the only
projective tangent channel is horizontal. The contained horizontal segment is
an exact direction and gives that channel weight one.
\end{proof}

\begin{remark}[Spatial variation]
Upper semicontinuity in Theorem~\ref{thm:directional-profile-usc} concerns
the direction at a fixed point. It does not imply spatial upper
semicontinuity of Point-Cross. For the semialgebraic cusp \(\mathcal K_2\),
the points \(x_j=(1/j,0)\), \(j\ge2\), are interior points and have
\(\dimcross\{x_j\}_{\mathcal K_2}=2\), whereas
\(\dimcross\{0\}_{\mathcal K_2}=1\).
\end{remark}

\subsection{Comparison of the four dimensions}
\label{subsec:separation-regimes}

The principal computations are collected below. To display the full
four-way comparison, we include the Point-Vector Dimension. The first row
uses an affine \(m\)-plane germ. For a general
curved \(C^1\) manifold, smooth saturation still gives
\(\dimcross=\Ptan=m\), whereas \(\Pvec\) need not equal \(m\).

\begin{table}[htbp]
\centering
\footnotesize
\setlength{\tabcolsep}{2.5pt}
\renewcommand{\arraystretch}{1.2}
\begin{tabular}{@{}>{\raggedright\arraybackslash}p{0.21\textwidth}cccc
                >{\raggedright\arraybackslash}p{0.25\textwidth}@{}}
\toprule
Germ at the indicated point
&\(\Pvec\)&\(\Pbox\)&\(\dimcross\)&\(\Ptan\)&Mechanism\\
\midrule
Affine \(m\)-plane germ
&\(m\)&\(m\)&\(m\)&\(m\)&Affine subspace\\

Cantor dust \(\mathcal D_K\)
&\(0\)&\(2d\)&\(2d\)&\(2\)&Product additivity\\

Cantor cross \(\mathcal X_K\)
&\(0\)&\(d\)&\(2d\)&\(2\)&Two fractional axial channels\\

Chirp \(A_\alpha\), \(0<\alpha<1\)
&\(0\)&\(\dfrac{3-\alpha}{2}\)&\(\dfrac32\)&\(2\)&Fractional finite-slope weights\\[0.5em]

Bedford\textendash{}McMullen carpet \(\Lambda\)
&\(1\)&\(2-d\)&\(1+d\)&\(2\)&Anisotropic self-affinity\\

Sierpiński vertex
&\(2\)&\(\substack{\log 3/\log 2\\ \text{or }\log 8/\log 3}\)
&\(2\)&\(2\)&Full side channels\\[0.6em]

Filled cusp \(\mathcal K_q\)
&\(1\)&\(2\)&\(1\)&\(1\)&Planar thickness in one channel\\
\bottomrule
\end{tabular}
\caption{Comparison of the Point-Vector, local upper box, Point-Cross, and
Point-Tangential Dimensions. Here
\(d=\log 2/\log 3\).}
\label{tab:separation-regimes}
\end{table}
\FloatBarrier

Three geometrically different families satisfy the strict four-way ordering
\(\Pvec<\Pbox<\dimcross<\Ptan\): the Cantor cross, the chirps with
\(0<\alpha<1\), and the Bedford\textendash{}McMullen carpet. The mechanisms are,
respectively, separated arms, oscillatory contact restrictions, and
anisotropic self-affinity. In the first two cases \(\Pvec=0\), whereas the
carpet retains one exact segment direction. The Cantor dust satisfies
\(\Pvec<\Pbox=\dimcross<\Ptan\), while an affine subspace has equality
of all four quantities. The Sierpiński vertices satisfy
\(\Pbox<\Pvec=\dimcross=\Ptan\), and the filled cusp gives
\(\Pvec=\dimcross=\Ptan<\Pbox\). Thus the directional hierarchy
\(\Pvec\le\dimcross\le\Ptan\) coexists with a local covering quantity that
need not be comparable with its middle term.

\section{Discussion and open problems}
\label{sec:outlook-open-problems}

The weighted profile records how much local upper box complexity a regular
curve can capture in a prescribed direction. Its threshold subspaces retain
the independent directions available at each weight, and Point-Cross is
the integral of their ranks. Reconstruction makes this profile stable under
closure and covariant in the \(C^1\) category. The fractional examples
show where it contains information beyond the unweighted tangent support.
Several questions remain.

\paragraph{Computing the weights.}
Projection bounds determine the scalar value of the carpet considered here,
but not all of its oblique weights. Can its symbolic structure determine
those weights? More generally, which self-affine germs admit a formula for
the full profile? The chirp calculation suggests a different route through
uniform contact counts with regular curves.

\paragraph{Continuity and attainment.}
The profile is upper semicontinuous, but the chirps exhibit jumps and
Example~\ref{ex:scalar-nonattainment} has no maximizing independent family
despite continuity on its support. Which geometric hypotheses imply
continuity, or prevent near-maximizing families from losing independence?
Spatial regularity is a separate issue, as the filled cusp shows.

\paragraph{Bi-Lipschitz changes of coordinates.}
Scalar invariance fails for general bi-Lipschitz maps. The map
\((u,w)\mapsto(u,w-|u|)\) sends the graph of \(|u|\), with Point-Cross
value two at zero, to a line with value one. Which bounds survive under
restricted bi-Lipschitz equivalences? Sequence-selection geometry and
directional bundles offer classes in which to test this
question~\cite{KoikePaunescu2015,KoikePaunescu2024}.

\paragraph{Joint realizations.}
Axial unions show that independent weights alone do not force an additive
lower bound for local upper box dimension. Which additional geometric
conditions ensure that witnesses in several directions combine at common
scales? A useful criterion would recover product-type covering estimates
without assuming a Cartesian product in advance.

\paragraph{Declaration of generative AI use.}
The author used ChatGPT and Prism to assist with critical review, bibliographic
exploration, discussion of proof strategies, and revision of the exposition.
The author checked the resulting arguments and references and is responsible
for the mathematical content and the final text.

\end{document}